\documentclass[12pt]{amsart}

\usepackage{graphicx}
\usepackage[utf8]{inputenc}
\usepackage[english]{babel}
\usepackage{csquotes}
\usepackage{amssymb}
\usepackage{tikz-cd}

\usepackage{amsmath,amsfonts,pictex,graphicx,fullpage,etex,tikz,hyperref}
\usepackage{enumerate}
\usepackage{enumitem}

\newtheorem{theorem}{Theorem}[section]
\newtheorem{lemma}[theorem]{Lemma}
\newtheorem{proposition}[theorem]{Proposition}
\newtheorem{corollary}[theorem]{Corollary}

\newtheorem{StandingAssumption}[theorem]{Standing Assumption}

\theoremstyle{definition}
\newtheorem{definition}[theorem]{Definition}
\newtheorem{conjecture}[theorem]{Conjecture}

\theoremstyle{remark}
\newtheorem{remark}[theorem]{Remark}

\numberwithin{equation}{section}

\makeatletter

\newcommand{\T}{\mathrm{T}}

\newcommand{\R}{\mathbb{R}}

\newcommand{\Z}{\mathbb{Z}}

\newcommand{\x}{\mathbf{x}}

\newcommand{\I}{\mathrm{I}}
\newcommand{\N}{\mathbb{N}}

\newcommand{\SO}{\mathrm{SO}}
\newcommand{\SL}{\mathrm{SL}}

\newcommand{\inv}{^{-1}}

\newcommand{\vect}[1]{\boldsymbol{\mathbf{#1}}}
\newcommand{\Bad}{\mathrm{\mathbf{Bad}}}
\newcommand{\Bd}{\mathrm{\mathbf{Bd}}}

\newcommand{\Par}{\mathrm{\mathbf{Par}}}
\newcommand{\diag}{\mathrm{diag}}
\newcommand{\Span}{\mathrm{Span}}
\newcommand{\fk}{\mathfrak{k}}
\newcommand{\cL}{\mathcal{L}}
\newcommand{\cW}{\mathcal{W}}
\newcommand{\bphi}{\boldsymbol{\mathbf{\varphi}}}

\makeatother

\begin{document}

\title{Badly approximable points on manifolds and unipotent orbits in homogeneous spaces}

\author{Lei Yang }
\address{College of Mathematics, Sichuan University, Chengdu, Sichuan, 610065, China}

%\curraddr{College of Mathematics, Sichuan University, Chengdu, Sichuan, 610065, China.}
\email{lyang861028@gmail.com}
%\thanks{}
\thanks{The author is supported in part by ISF grant 2095/15, ERC grant AdG 267259, NSFC grant 11743006 and a start up research funding from Sichuan University}

%    General info

\date{}

%\dedicatory{}


\begin{abstract}
In this paper, we study the weighted $n$-dimensional badly approximable points on manifolds. Given a $C^n$ differentiable non-degenerate submanifold $\mathcal{U} \subset \R^n$, we will show that any countable intersection of the sets of the weighted badly approximable points on $\mathcal{U}$ has full Hausdorff dimension. This strengthens a result of Beresnevich \cite{Beresnevich} by removing the condition on weights and weakening the smoothness condition on manifolds. Compared to the work of Beresnevich, our approach relies on homogeneous dynamics. It turns out that in order to solve this problem, it is crucial to study the distribution of long pieces of unipotent orbits in homogeneous spaces. The proof relies on the linearization technique and representations of $\SL(n+1,\R)$.
\end{abstract}
\maketitle
%%%%%%%%%%%%%%%%%%%%%%%%%%%%%%%%%%%%%%%%%%%%%%%%%%%%%%%%%%%%%%%%%%%%%%%%%%%%%%%%%%%%%%%%%%%%%%%%%%%%%%%%%%%%%%%%%%
%%%%%%%%%%%%%%%%%%%%%%%%%%%%%%%%%%%%%%%%%%%%%%%%%%%%%%%%%%%%%%%%%%%%%%%%%%%%%%%%%%%%%%%%%%%%%%%%%%%%%%%%%%%%%%%%%%
%%%%%%%%%%%%%%%%%%%%%%%%%%%%%%%%%%%%%%%%%%% Introduction  %%%%%%%%%%%%%%%%%%%%%%%%%%%%%%%%%%%%%%%%%%%%%%%%%%%%%%%%
%%%%%%%%%%%%%%%%%%%%%%%%%%%%%%%%%%%%%%%%%%%%%%%%%%%%%%%%%%%%%%%%%%%%%%%%%%%%%%%%%%%%%%%%%%%%%%%%%%%%%%%%%%%%%%%%%%
%%%%%%%%%%%%%%%%%%%%%%%%%%%%%%%%%%%%%%%%%%%%%%%%%%%%%%%%%%%%%%%%%%%%%%%%%%%%%%%%%%%%%%%%%%%%%%%%%%%%%%%%%%%%%%%%%%
\section{Introduction}
\label{sec-introduction}
%%%%%%%%%%%%%%%%%%%%%%%%%%%%%%%%%%%%%%%%%%%%%%%%%%%%%%%%%%%%%%%%%%%%%%%%%%%%%%%%%%%%%%%%%%%%%%%%%%%%%%%%%%%%%%%%%%
%%%%%%%%%%%%%%%%%%%%%%%%%%%%%%%%%%%%%%%%%%%%%%%%%%%%%%%%%%%%%%%%%%%%%%%%%%%%%%%%%%%%%%%%%%%%%%%%%%%%%%%%%%%%%%%%%%
%%%%%%%%%%%%%%%%%%%%%%%%%% subsection: badly approximable vectors  %%%%%%%%%%%%%%%%%%%%%%%%%%%%%%%%%%%%%%%%%%%%%%%
%%%%%%%%%%%%%%%%%%%%%%%%%%%%%%%%%%%%%%%%%%%%%%%%%%%%%%%%%%%%%%%%%%%%%%%%%%%%%%%%%%%%%%%%%%%%%%%%%%%%%%%%%%%%%%%%%%
%%%%%%%%%%%%%%%%%%%%%%%%%%%%%%%%%%%%%%%%%%%%%%%%%%%%%%%%%%%%%%%%%%%%%%%%%%%%%%%%%%%%%%%%%%%%%%%%%%%%%%%%%%%%%%%%%%

\subsection{Badly approximable vectors}
\label{subsec-badly-approximable-vectors}
Given a positive integer $n$, a vector $\vect{r} = (r_1 , \dots, r_n)$ is called a $n$-dimensional weight if $r_i \geq 0$ for $i=1,\dots, n$ and 
\[r_1 + \cdots + r_n = 1.\] 
The weighted version of Dirichlet's approximation theorem says the following:
\begin{theorem}[Dirichlet's Theorem, 1842]
For any $n$-dimensional weight $\vect{r} = (r_1,\dots, r_n)$, the following statement holds. For any vector $\vect{x} = (x_1, \dots , x_n) \in \R^n$ and any $N > 1$, there exists an integer vector $\vect{p} = (p_1, \dots, p_n , q) \in \Z^{n+1}$such that $0 < |q| \leq N$ and 
\[ |q x_i + p_i| \leq N^{-r_i} \text{, for } i =1,\dots , n.\]
\end{theorem}
This theorem is the starting point of simutaneous Diophantine approximation. Using this theorem, one can easily show the following:
\begin{corollary}
\label{cor:dirichlet-diophantine-approximation}
For any vector $\x = (x_1, \dots, x_n) \in \R^n$, there are infinitely many integer vectors $\vect{p} = (p_1,\dots, p_n , q)\in \Z^{n+1}$ with $q \neq 0$ satisfying the following:
\begin{equation}
\label{equ:weighted-diophantine}
 |q|^{r_i} |q x_i + p_i| \leq 1 \text{ for } i = 1 ,\dots, n. 
\end{equation}
\end{corollary}
\par For almost every vector $\x \in \R^n$, the above corollary remains true if we replace $1$ with any smaller constant $c>0$ on the right hand side of \eqref{equ:weighted-diophantine}, see \cite{Daven_Schm} and \cite{Klein_Weiss}. The exceptional vectors are called $\vect{r}$-weighted badly approximable vectors. We give the formal definition as follows:
\begin{definition}
\label{def:badly-approximable}
Given an $n$-dimensional weight $\vect{r} = (r_1, \dots, r_n)$, a vector $\x \in \R^n$ is called $\vect{r}$-weighted badly approximable if there exists a constant $c >0$ such that for any $\vect{p} =(p_1, \dots, p_n, q) \in \Z^{n+1}$ with $q \neq 0$,
\[
\max_{1 \leq i \leq n} |q|^{r_i} |q x_i + p_i| \geq c .
\] 
\end{definition}
For an $n$-dimensional weight $\vect{r}$, let us denote the set of $\vect{r}$-weighted badly approximable vectors in $\R^n$ by $\Bad(\vect{r})$. In particular, $\Bad(1)$ denotes the set of badly approximable numbers.
\par $\Bad(\vect{r})$ is a fundamental object in metric Diophantine approximation. The study of its properties has a long history and attracts people from both number theory and homogeneous dynamics. In view of \cite{Klein_Weiss}, we know that the Lebesgue measure of $\Bad(\vect{r})$ is zero. However, it turns out that every $\Bad(\vect{r})$ has full Hausdorff dimension, cf. \cite{Jarnik}, \cite{Schmidt-Game}, \cite{Pollington-Velani} and \cite{Klein-Weiss-Modified-Schmidt}. The intersections of $\Bad(\vect{r})$ with different weights $\vect{r}$ have been of major interest for several decades. In particular, Wolfgang M. Schmidt conjectured the following:
\begin{conjecture}[Schmidt's Conjecture, see \cite{Schmidt}]
\label{conj:schmidts-conjecture}
\par For $n =2$,
\[\Bad( 1/3, 2/3) \cap \Bad( 2/3, 1/3) \neq \emptyset.\]
\end{conjecture}
In 2011, Badziahin, Pollington and Velani \cite{BPV} settled this conjecture by showing the following: for any countable collection of $2$-dimensional weights $\{(i_t, j_t): t  \in \N\}$, if $\liminf_{t \to \infty} \min\{i_t, j_t\} >0$, then 
\[\dim_H\left(\bigcap_{t =1}^{\infty} \Bad( i_t, j_t )\right) = 2,\]
where $\dim_H(\cdot)$ denotes the Hausdorff dimension of a set. An (see \cite{An-2013-BLMS} and \cite{An}) later strengthens their result by removing the condition on the weights. In fact, in \cite{An}, An proves the following much stronger result: for any $2$-dimensional weight $(r_1, r_2)$, $\Bad(r_1, r_2)$ is $(24\sqrt{2})\inv$-winning. Here a set is called $\alpha$-winning if it is a winning set for Schmidt's $(\alpha, \beta)$-game for any $\beta \in (0,1)$. This statement implies that any countable intersection of sets of weighted badly approximable vectors is $\alpha$-winning. Nesharim and Simmons \cite{Nesharim-Simmons} further show that every $\Bad(r_1, r_2)$ is hyperplane absolute winning. The reader is referred to \cite{Schmidt-Game} for more details of Schmidt's game and to \cite{broderick_fishman_kleinbock_reich_weiss_2012} for details about hyperplane winning sets.
\par Badly approximable vectors lying on planar curves are studied by An, Beresnevich and Velani \cite{An-Beresnevich-Velani}. They prove that for any non-degenerate planar curve $\mathcal{C}$ and any weight $(r_1, r_2)$, $\Bad(r_1, r_2)\cap \mathcal{C}$ is $\frac{1}{2}$-winning. 
\par For $n \geq 3$, the problem turns out to be essentially more difficult. Beresnevich \cite{Beresnevich} makes the first breakthrough:
\begin{theorem}[see~{\cite[Corollary 1]{Beresnevich}}]
\label{thm:beresnevich-thm}
Let $n \geq 2$ be an integer and $\mathcal{U} \subset \R^n$ be an analytic and non-degenerate submanifold in $\R^n$. Let $W$ be a finite or countable set of $n$-dimensional weights such that $\inf_{\vect{r} \in W}\{ \tau(\vect{r})\} >0$ where $\tau(r_1, \dots, r_n) := \min\{ r_i : r_i >0\}$ for an $n$-dimensional weight $(r_1, \dots, r_n)$. Then 
\[ \dim_H \left( \bigcap_{\vect{r} \in W} \Bad(\vect{r}) \cap \mathcal{U} \right) = \dim \mathcal{U}. \] 
\end{theorem}
\begin{remark}
\par Here a submanifold is called non-degenerate if the derivatives at each point span the whole space. In the setting of analytic submanifolds, this is equivalent to that the submanifold is not contained in any hyperplane of $\R^n$.
\end{remark}

%%%%%%%%%%%%%%%%%%%%%%%%%%%%%%%%%%%%%%%%%%%%%%%%%%%%%%%%%%%%%%%%%%%%%%%%%%%%%%%%%%%%%%%%%%%%%%%%%%%%%%%%%%%%%%%%%
%%%%%%%%%%%%%%%%%%%%%%%%%%%%%%%%%%%%%%%%%%%%%%%%%%%%%%%%%%%%%%%%%%%%%%%%%%%%%%%%%%%%%%%%%%%%%%%%%%%%%%%%%%%%%%%%%
%%%%%%%%%%%%%%%%%%%%%%%%%%%%%%%%%%%%%%%%%%   notation    %%%%%%%%%%%%%%%%%%%%%%%%%%%%%%%%%%%%%%%%%%%%%%%%%%%%%%%%
%%%%%%%%%%%%%%%%%%%%%%%%%%%%%%%%%%%%%%%%%%%%%%%%%%%%%%%%%%%%%%%%%%%%%%%%%%%%%%%%%%%%%%%%%%%%%%%%%%%%%%%%%%%%%%%%%
%%%%%%%%%%%%%%%%%%%%%%%%%%%%%%%%%%%%%%%%%%%%%%%%%%%%%%%%%%%%%%%%%%%%%%%%%%%%%%%%%%%%%%%%%%%%%%%%%%%%%%%%%%%%%%%%%
\subsection{Notation}
\label{subsec-notation}
\par In this paper, we will fix the following notation.
\par For a set $\mathcal{S}$, let $\sharp\mathcal{S}$ denote the cardinality of $\mathcal{S}$. For a measurable subset $E \subset \R$, let $m(E)$ denote its Lebesgue measure.
\par For a matrix $M$, let $M^{\T}$ denote its transpose. For integer $k >0$, let $\I_k$ denote the $k$-dimensional identity matrix.
\par Let $\|\cdot\|$ denote the supremum norm on $\R^n$ and $\R^{n+1}$. Let $\|\cdot\|_2$ denote the Euclidean norm on $\R^n$ and $\R^{n+1}$. For $\vect{x} \in \R^{n+1}$ (or $\in \R^n$) and $r >0$, let $B(\vect{x}, r)$ denote the closed ball in $\R^{n+1}$ (or $\R^n$) centered at $\vect{x}$ of radius $r$, with respect to $\|\cdot\|$. For every $i = 1, \dots , n+1$, there is a natural supremum norm on $\bigwedge^i \R^{n+1}$. Let us denote it by $\|\cdot\|$.
\par Throughout this paper, when we say that $C$ is a constant, we always mean that $c$ is a constant only depending on the dimension $n$. For quantities $A$ and $B$, let us use $A \ll B$ to mean that there is a constant $C>0$ such that $A \leq C B$. Let $A \asymp B$ mean that $A \ll B$ and $B \ll A$. For a quantity $A$, let $O(A)$ denote a quantity which is $\ll A$ or a vector whose norm is $\ll A$.

%%%%%%%%%%%%%%%%%%%%%%%%%%%%%%%%%%%%%%%%%%%%%%%%%%%%%%%%%%%%%%%%%%%%%%%%%%%%%%%%%%%%%%%%%%%%%%%%%%%%%%%%%%%%%%%%%
%%%%%%%%%%%%%%%%%%%%%%%%%%%%%%%%%%%%%%%%%%%%%%%%%%%%%%%%%%%%%%%%%%%%%%%%%%%%%%%%%%%%%%%%%%%%%%%%%%%%%%%%%%%%%%%%%
%%%%%%%%%%%%%%%%%%%%%%%%%%%%%%%%%%%%%%%%%   Main result   %%%%%%%%%%%%%%%%%%%%%%%%%%%%%%%%%%%%%%%%%%%%%%%%%%%%%%%
%%%%%%%%%%%%%%%%%%%%%%%%%%%%%%%%%%%%%%%%%%%%%%%%%%%%%%%%%%%%%%%%%%%%%%%%%%%%%%%%%%%%%%%%%%%%%%%%%%%%%%%%%%%%%%%%%
%%%%%%%%%%%%%%%%%%%%%%%%%%%%%%%%%%%%%%%%%%%%%%%%%%%%%%%%%%%%%%%%%%%%%%%%%%%%%%%%%%%%%%%%%%%%%%%%%%%%%%%%%%%%%%%%%
\subsection{Main results}
\label{subsec-main-result}
In this paper, we will strengthen Theorem \ref{thm:beresnevich-thm} by removing the condition on weights and weakening the analytic condition to differentiable condition on submanifolds.

\par To simplify the exposition, in this paper, we will focus on the case of curves:
\begin{theorem}
\label{thm:main-thm}
Let $\bphi: I = [a,b] \to \R^n$ be a $C^n$ differentiable and non-degenerate curve in $\R^n$. Let $W$ be a finite or countable set of $n$-dimensional weights. Then 
\[ \dim_H \left( \bigcap_{\vect{r} \in W} \Bad(\vect{r}) \cap \bphi(I) \right) = 1. \] 
\end{theorem}
\par The proof for curves directly applies to any $C^n$ non-degenerate manifolds, see \S \ref{subsec-general-case} for detailed explanation. Therefore, Theorem \ref{thm:main-thm} holds for any $C^n$ non-degenerate manifolds. In Theorem \ref{thm:beresnevich-thm}, the analyticity condition comes from a fiber lemma (cf. \cite[Appendix C]{Beresnevich}) which reduces the general case to the case of curves.
\par In fact, we can prove the following stronger statement:
\begin{theorem}
\label{thm:double-intersection}
Let $W$ be a finite or countable set of $n$-dimensional weights and $\mathcal{F}_n(B)$ be a finite family of $C^n$ differentiable non-degenerate maps $\bphi : [0,1] \to \R^n$. Then 
\[ \dim_H \left( \bigcap_{\bphi \in \mathcal{F}_n(B)} \bigcap_{\vect{r} \in W} \bphi\inv(\Bad(\vect{r}))\right) = 1.\]
\end{theorem}
 For the same reason as above, this statement holds when $[0,1]$ is replaced by a $m$-dimensional ball $B \subset \R^m$ for any $m \leq n$. 
\par Compared with \cite{Beresnevich}, in this paper, we study this problem through homogeneous dynamics and prove Theorem \ref{thm:main-thm} and \ref{thm:double-intersection} using the linearization technique.
%%%%%%%%%%%%%%%%%%%%%%%%%%%%%%%%%%%%%%%%%%%%%%%%%%%%%%%%%%%%%%%%%%%%%%%%%%%%%%%%%%%%%%%%%%%%%%%%%%%%%%%%%%%%%%%
%%%%%%%%%%%%%%%%%%%%%%%%%%%%%%%%%%%%%%%%%%%%%%%%%%%%%%%%%%%%%%%%%%%%%%%%%%%%%%%%%%%%%%%%%%%%%%%%%%%%%%%%%%%%%%%
%%%%%%%%%%%%%%%%%%%%%%%%%%%%%%%%%%%%%%%%% Homogeneous spaces %%%%%%%%%%%%%%%%%%%%%%%%%%%%%%%%%%%%%%%%%%%%%%%%%%
%%%%%%%%%%%%%%%%%%%%%%%%%%%%%%%%%%%%%%%%%%%%%%%%%%%%%%%%%%%%%%%%%%%%%%%%%%%%%%%%%%%%%%%%%%%%%%%%%%%%%%%%%%%%%%%
%%%%%%%%%%%%%%%%%%%%%%%%%%%%%%%%%%%%%%%%%%%%%%%%%%%%%%%%%%%%%%%%%%%%%%%%%%%%%%%%%%%%%%%%%%%%%%%%%%%%%%%%%%%%%%%
\subsection{Bounded orbits in homogeneous spaces}
\label{subsec-bounded-orbits-homogeneous-space}
\par Let us briefly recall the correspondence between Diophantine approximation and homogeneous dynamics. The reader may see \cite{Dani}, \cite{Klein_Mar} and \cite{Klein_Weiss} for more details.
\par Let $G = \SL(n+1, \R)$, and $\Gamma  = \SL(n+1, \Z)$. The homogeneous space $X = G/\Gamma$ can be identified with the space of unimodular lattices in $\R^{n+1}$. For any $g \in \SL(n+1, \R)$, the point $g\Gamma$ is identified with the lattice $g\Z^{n+1}$. For $\epsilon >0$, let us define
\begin{equation}
\label{equ:compact-subset-homo-space}
K_{\epsilon}:=\left\{\Lambda \in X: \Lambda \cap B(\vect{0}, \epsilon) = \{\vect{0}\}\right\}.
\end{equation}
 By Mahler's compactness criterion \cite{Mahler1946}, every $K_{\epsilon}$ is a compact subset of $X$ and every compact subset of $X$ is contained in some $K_{\epsilon}$.
\par For a weight $\vect{r}=(r_1, \dots, r_n)$, let us define the diagonal subgroup $A_{\vect{r}} \subset G$ as follows:
\[A_{\vect{r}} := \left\{ a_{\vect{r}}(t) := \begin{bmatrix}e^{r_1 t} & & & \\ & \ddots & & \\ & & e^{r_n t} & \\ & & & e^{-t}\end{bmatrix}: t \in \R \right\}.\]
\par For $\x \in \R^n$, let us denote 
\[V(\x) := \begin{bmatrix} \I_n & \x \\ & 1 \end{bmatrix}.\]
\begin{proposition}[{\cite[Theorem 1.5]{Kleinbock}}]
\label{prop:correspondence-diophantine-dynamics}
\par $\x \in \Bad(\vect{r})$ if and only if $\{a_{\vect{r}}(t)V(\x)\Z^{n+1}: t >0\}$ is bounded.
\end{proposition}

Therefore our main theorem is equivalent to saying that for any $C^n$ non-degenerate submanifold $\mathcal{U} \subset \R^n$ and any countable collection of one-parameter diagonal subgroups $\{A_{\vect{r}_s}: s \in \N\}$, the set of $\x \in \mathcal{U}$ such that 
\[\{a_{\vect{r}_s}(t)V(\x)\Z^{n+1}: t>0\}\] 
is bounded for all $s \in \N$ has full Hausdorff dimension.

\par The study of bounded trajectories under the action of diagonal subgroups in homogeneous spaces is a fundamental topic in homogeneous dynamics and has been active for decades. The basic set up of this type of problems is the following. Let $G$ be a Lie group and $\Gamma \subset G$ be a nonuniform lattice in $G$. Then $X = G/\Gamma$ is a noncompact homogeneous space. Let $A = \{ a(t) : t \in \R\}$ be a one-dimensional diagonalizable subgroup and let $\Bd(A)$ be the set of $x \in X$ such that $A^{+}x$ is bounded in $X$, where $A^+ := \{a(t): t >0\}$. Then one can ask whether $\Bd(A)$ has full Hausdorff dimension. For a submanifold $\mathcal{U} \subset X $, one can also ask whether $\Bd(A)\cap \mathcal{U}$ has Hausdorff dimension $\dim \mathcal{U}$.
\par  In 1986, Dani \cite{Dani-Bounded-Orbits} studies the case where $G$ is a semisimple Lie group with $\R$-rank one. In this case, he proves that for any non-quasi-unipotent one parameter subgroup $A \subset G$, $\Bd(A)$ has full Hausdorff dimension. His proof relies on Schmidt's game. In 1996, Kleinbock and Margulis \cite{Klein-Margulis-Bounded} study the case where $G$ is a semisimple Lie group and $\Gamma$ is a irreducible lattice in $G$. In this case, they prove that $\Bd(A)$ has full Hausdorff dimension for any non-quasi-unipotent subgroup $A$. Their proof is based on the mixing property of the action of $A$ on $X$. Recently, An, Guan and Kleinbock study the case where $G = \SL(3,\R)$ and $\Gamma = \SL(3,\Z)$. They prove that for any countable collection of diagonalizable one-parameter subgroups $\{F_s : s \in \N\}$, the intersection $\bigcap_{s =1 }^{\infty} \Bd(F_s)$ has full Hausdorff dimension. Their proof closely follows the argument in the work of An \cite{An} and uses a variantion of Schmidt's game.
%%%%%%%%%%%%%%%%%%%%%%%%%%%%%%%%%%%%%%%%%%%%%%%%%%%%%%%%%%%%%%%%%%%%%%%%%%%%%%%%%%%%%%%%%%%%%%%%%%%%%%%%%%%%%%%%%%
%%%%%%%%%%%%%%%%%%%%%%%%%%%%%%%%%%%%%%%%%%%%%%%%%%%%%%%%%%%%%%%%%%%%%%%%%%%%%%%%%%%%%%%%%%%%%%%%%%%%%%%%%%%%%%%%%%
%%%%%%%%%%%%%%%%%%%%%%%%%%%%%%%%%%%%%%%% The linearization technique %%%%%%%%%%%%%%%%%%%%%%%%%%%%%%%%%%%%%%%%%%%%%
%%%%%%%%%%%%%%%%%%%%%%%%%%%%%%%%%%%%%%%%%%%%%%%%%%%%%%%%%%%%%%%%%%%%%%%%%%%%%%%%%%%%%%%%%%%%%%%%%%%%%%%%%%%%%%%%%%
%%%%%%%%%%%%%%%%%%%%%%%%%%%%%%%%%%%%%%%%%%%%%%%%%%%%%%%%%%%%%%%%%%%%%%%%%%%%%%%%%%%%%%%%%%%%%%%%%%%%%%%%%%%%%%%%%%
\subsection{The linearization technique}
\label{subsec-main-tools-proof}
\par In \cite{Beresnevich}, the proof relies on the theory of geometry of numbers. In this paper, we study this problem through homogeneous dynamics and tackle the technical difficulties using the linearization technique. It turns out that in order to get full Hausdorff dimension, it is crucial to study distributions of long pieces of unipotent orbits in the homogeneous space $G/\Gamma$. To be specific, for a particular long piece $C$ of a unipotent orbit, we need to estimate the length of the part in $C$ staying outside a large compact subset $K$ of $G/\Gamma$. In homogeneous dynamics, the standard tool to study this type of problem is the linearization technique. The linearization technique is a standard and powerful technique in homogeneous dynamics. Using the linearization technique, we can transform a problem in dynamical systems to a problem on linear representations. Then we can study this problem using tools and results in representation theory. 
\par Let us briefly describe the technical difficulty when we apply the linearization technique. Let $\mathcal{V}$ be a finite dimensional linear representation of $\SL(n+1,\R)$ with a norm $\|\cdot\|$ and $\Gamma(\mathcal{V}) \subset \mathcal{V}$ be a fixed discrete subset of $\mathcal{V}$. Let $U = \{u(r): r \in \R\}$ be a one parameter unipotent subgroup of $G$. Given a large number $ T > 1$, we want to estimate the measure of $r \in [ - T, T]$ such that there exists $v \in \Gamma(\mathcal{V})$ such that $\|u(r)v\| \leq \epsilon$ where $\epsilon>0$ is a small number. By Dani-Margulis non-divergence theorem (see \cite{Dani-Margulis}), the measure is very small compared with $T$ given that for any such $v \in \Gamma(\mathcal{V})$ 
\[ \max\{ \|u(r)v\|: r \in [-T, T] \} \geq \rho \]
where $\rho >0$ is some fixed number. The difficulty is to handle the case where there exists some $v \in \Gamma(\mathcal{V})$, such that 
\[ \max\{ \|u(r)v\|: r \in [-T, T] \} < \rho. \] 
 Let us call such intervals $T$-bad intervals. In this paper, we will use representation theory to study properties of such $v$'s. We then use these properties to show that in a longer interval, say $[-T^2, T^2]$, the number of $T$-bad intervals is $\ll T^{1-\mu}$ for some constant $\mu >0$. This result is sufficient to prove Theorem \ref{thm:main-thm}.  
 \par In this paper, $\mathcal{V}$ is the canonical representation of $\SL(n+1, \R)$ on $\bigwedge^i \R^{n+1}$ and $\Gamma(\mathcal{V}) = \bigwedge^i \Z^{n+1} \setminus \{\vect{0}\}$ where $i = 1, \dots, n$.
\par The main technical results in this paper are proved in \S \ref{sec:count-dangerous-intervals}, \S \ref{subsec-dangerous-case} and \S \ref{subsec-extremely-dangerous-case}.
\par We refer the reader to \cite{Ratner}, \cite{Margulis-Tomanov}, \cite{Mozes_Shah}, \cite{Shah_1}, \cite{Shah_2} and \cite{Lindenstrauss-Margulis} for more applications of the linearization technique.
%%%%%%%%%%%%%%%%%%%%%%%%%%%%%%%%%%%%%%%%%%%%%%%%%%%%%%%%%%%%%%%%%%%%%%%%%%%%%%%%%%%%%%%%%%%%%%%%%%%%%%%%%%%%%%%%%%
%%%%%%%%%%%%%%%%%%%%%%%%%%%%%%%%%%%%%%%%%%%%%%%%%%%%%%%%%%%%%%%%%%%%%%%%%%%%%%%%%%%%%%%%%%%%%%%%%%%%%%%%%%%%%%%%%%
%%%%%%%%%%%%%%%%%%%%%%%%%%%%%%%%%%%%% Organization of the paper %%%%%%%%%%%%%%%%%%%%%%%%%%%%%%%%%%%%%%%%%%%%%%%%%%
%%%%%%%%%%%%%%%%%%%%%%%%%%%%%%%%%%%%%%%%%%%%%%%%%%%%%%%%%%%%%%%%%%%%%%%%%%%%%%%%%%%%%%%%%%%%%%%%%%%%%%%%%%%%%%%%%%
%%%%%%%%%%%%%%%%%%%%%%%%%%%%%%%%%%%%%%%%%%%%%%%%%%%%%%%%%%%%%%%%%%%%%%%%%%%%%%%%%%%%%%%%%%%%%%%%%%%%%%%%%%%%%%%%%%

\subsection{The organization of the paper}
\label{subsec-organization-of-paper}
\par The paper is organized as follows:
\begin{itemize}
\item In \S \ref{sec-preliminaries}, we will recall some basic facts on Diophantine approximation, linear representations and lattices in $\R^{n+1}$.
\item In \S \ref{sec:cantor-like-construction}, we will recall a theorem on computing the Hausdorff dimension of Cantor like sets. We will also construct a Cantor-like covering of the set of weighted badly approximable points.
\item In \S \ref{sec:count-dangerous-intervals}, we will prove two technical results on counting lattice points. Proposition \ref{prop:count-dangerous-intervals} is one of the main technical contributions of this paper. Its proof relies on the linearization technique and $\SL(n+1, \R)$ representations.
\item In \S \ref{sec-proof-main-result}, we will give the proof of Proposition \ref{prop:distance-sequence-small}, which implies Theorem \ref{thm:main-task}, \ref{thm:main-thm} and \ref{thm:double-intersection}. We split the proof into three parts: {\bf the generic case}, {\bf the dangerous case} and {\bf the extremely dangerous case}. \S \ref{subsec-generic-case} handles {\bf the generic case}. The proof relies on the Dani-Margulis non-divergence theorem (Theorem \ref{thm:non-divergence-dani-margulis}). \S \ref{subsec-dangerous-case} handles {\bf the dangerous case}. The proof relies on Proposition \ref{prop:count-dangerous-intervals} proved in \S \ref{sec:count-dangerous-intervals} and the linearization technique. \S \ref{subsec-extremely-dangerous-case} handles {\bf the extremely dangerous case}. The proof relies on Proposition \ref{prop:count-extremely-dangerous-intervals} proved in \ref{sec:count-dangerous-intervals} and the linearization technique. Finally, we will explain how to adapt the proof to handle general $C^n$ non-degenerate manifolds.
\end{itemize}

%%%%%%%%%%%%%%%%%%%%%%%%%%%%%%%%%%%%%%%%%%%%%%%%%%%%%%%%%%%%%%%%%%%%%%%%%%%%%%%%%%%%%%%%%%%%%%%%%%%%%%%%%%%%%%%%%
%%%%%%%%%%%%%%%%%%%%%%%%%%%%%%%%%%%%%%%%%%%%%%%%%%%%%%%%%%%%%%%%%%%%%%%%%%%%%%%%%%%%%%%%%%%%%%%%%%%%%%%%%%%%%%%%%
%%%%%%%%%%%%%%%%%%%%%%%%%%%%%%%%%%%%%% Acknowledgements %%%%%%%%%%%%%%%%%%%%%%%%%%%%%%%%%%%%%%%%%%%%%%%%%%%%%%%%%
%%%%%%%%%%%%%%%%%%%%%%%%%%%%%%%%%%%%%%%%%%%%%%%%%%%%%%%%%%%%%%%%%%%%%%%%%%%%%%%%%%%%%%%%%%%%%%%%%%%%%%%%%%%%%%%%%
%%%%%%%%%%%%%%%%%%%%%%%%%%%%%%%%%%%%%%%%%%%%%%%%%%%%%%%%%%%%%%%%%%%%%%%%%%%%%%%%%%%%%%%%%%%%%%%%%%%%%%%%%%%%%%%%%
\medskip
\par \noindent {\bf Acknowledgements.} The author would like to thank Elon Lindenstrauss and Barak Weiss for sharing many insightful ideas on this problem. He also thanks Shahar Mozes for helpful conversations on this problem. He appreciates their encouragements during the process of this work. He thanks Victor Beresnevich for inspiring discussion on this topic, especially for pointing out that the proof works for $C^n$ differentiable submanifolds. He also thanks Jinpeng An, Anish Ghosh, Erez Nesharim and Sanju Velani for their interests and helpful comments on an earlier version of this paper. Thanks are due to the anonymous referees for carefully reading the paper and giving many valuable suggestions that led to this revised version.

%%%%%%%%%%%%%%%%%%%%%%%%%%%%%%%%%%%%%%%%%%%%%%%%%%%%%%%%%%%%%%%%%%%%%%%%%%%%%%%%%%%%%%%%%%%%%%%%%%%%%%%%%%%%%%%%%%
%%%%%%%%%%%%%%%%%%%%%%%%%%%%%%%%%%%%%%%%%%%%%%%%%%%%%%%%%%%%%%%%%%%%%%%%%%%%%%%%%%%%%%%%%%%%%%%%%%%%%%%%%%%%%%%%%%
%%%%%%%%%%%%%%%%%%%%%%%%%%%%%%%%%%%%%%%%%   Preliminaries  %%%%%%%%%%%%%%%%%%%%%%%%%%%%%%%%%%%%%%%%%%%%%%%%%%%%%%%
%%%%%%%%%%%%%%%%%%%%%%%%%%%%%%%%%%%%%%%%%%%%%%%%%%%%%%%%%%%%%%%%%%%%%%%%%%%%%%%%%%%%%%%%%%%%%%%%%%%%%%%%%%%%%%%%%%
%%%%%%%%%%%%%%%%%%%%%%%%%%%%%%%%%%%%%%%%%%%%%%%%%%%%%%%%%%%%%%%%%%%%%%%%%%%%%%%%%%%%%%%%%%%%%%%%%%%%%%%%%%%%%%%%%%
%%%%%%%%%%%%%%%%%%%%%%%%%%%%%%%%%%%%%%%%%%%%%%%%%%%%%%%%%%%%%%%%%%%%%%%%%%%%%%%%%%%%%%%%%%%%%%%%%%%%%%%%%%%%%%%%%%

\section{Preliminaries}
\label{sec-preliminaries}
%%%%%%%%%%%%%%%%%%%%%%%%%%%%%%%%%%%%%%%%%%%%%%%%%%%%%%%%%%%%%%%%%%%%%%%%%%%%%%%%%%%%%%%%%%%%%%%%%%%%%%%%%%%%%%%%%%
%%%%%%%%%%%%%%%%%%%%%%%%%%%%%%%%%%%%%%%%%%%%%%%%%%%%%%%%%%%%%%%%%%%%%%%%%%%%%%%%%%%%%%%%%%%%%%%%%%%%%%%%%%%%%%%%%%
%%%%%%%%%%%%%%%%%%%%%%%%%%%%%%%%%%%% dual form of approximation %%%%%%%%%%%%%%%%%%%%%%%%%%%%%%%%%%%%%%%%%%%%%%%%%%
%%%%%%%%%%%%%%%%%%%%%%%%%%%%%%%%%%%%%%%%%%%%%%%%%%%%%%%%%%%%%%%%%%%%%%%%%%%%%%%%%%%%%%%%%%%%%%%%%%%%%%%%%%%%%%%%%%
%%%%%%%%%%%%%%%%%%%%%%%%%%%%%%%%%%%%%%%%%%%%%%%%%%%%%%%%%%%%%%%%%%%%%%%%%%%%%%%%%%%%%%%%%%%%%%%%%%%%%%%%%%%%%%%%%%
\subsection{Dual form of approximation}
\label{subsec-dual-form}
\par We first recall the following equivalent definition of $\Bad(\vect{r})$:
\begin{lemma}[see~{\cite[Lemma 1]{Beresnevich}}]
\label{lm:equivalent-definition-badly-approximable}
Let $\vect{r} = (r_1, \dots, r_n) \in \R^n$ be a weight and $\x \in \R^n$. The following statements are equivalent:
\begin{enumerate}
\item $\x \in \Bad(\vect{r})$.
\item There exists $c >0$ such that for any integer vector $(p_1, \dots, p_n ,q)$ such that $q \neq 0$, we have that 
\[ \max_{1 \leq i \leq n} |q|^{r_i}|q x_i + p_i| \geq c.\]
\item There exists $c >0$ such that for any $N \geq 1$, the only integer solution $(a_0, a_1, \dots, a_n)$ to the system
\[
	\begin{array}{rr}
	|a_0 + a_1 x_1 + \cdots + a_n x_n| < c N\inv , & |a_i| < N^{r_i} \text{ for all } 1 \leq i \leq n
	\end{array}
\]
is $a_0 = a_1 = \cdots = a_n =0$.
\end{enumerate}
\end{lemma}
\begin{proof}
The reader is referred to \cite{Mahler}, \cite[Appendix]{BPV} and \cite[Appendix A]{Beresnevich} for the proof.
\end{proof}
\par Later in this paper we will use the third statement as the definition of $\Bad(\vect{r})$.
\par Given a weight $\vect{r}=(r_1, \dots, r_n)$, let us define 
\[ D_{\vect{r}} := \left\{ d_{\vect{r}}(t) := \begin{bmatrix}e^{t} & & & \\
 & e^{-r_1 t} & & \\
  & & \ddots & \\
  & & & e^{-r_n t} \end{bmatrix}: t \in \R \right\}.\]
 For $\x \in \R^n$, let us define 
 \[U(\x) := \begin{bmatrix} 1 & \x^{\T} \\ & \I_n \end{bmatrix}.\]

If we use the third statement in Lemma \ref{lm:equivalent-definition-badly-approximable} as the definition of $\Bad(\vect{r})$, then in view of \cite[Theorem 1.5]{Kleinbock} we have that $\x \in \Bad(\vect{r})$ if and only if $U(\x)\Z^{n+1} \in \Bd(D_{\vect{r}})$.

%%%%%%%%%%%%%%%%%%%%%%%%%%%%%%%%%%%%%%%%%%%%%%%%%%%%%%%%%%%%%%%%%%%%%%%%%%%%%%%%%%%%%%%%%%%%%%%%%%%%%%%%%%%%%%%%%%%
%%%%%%%%%%%%%%%%%%%%%%%%%%%%%%%%%%%%%%%%%%%%%%%%%%%%%%%%%%%%%%%%%%%%%%%%%%%%%%%%%%%%%%%%%%%%%%%%%%%%%%%%%%%%%%%%%%%
%%%%%%%%%%%%%%%%%%%%%%%%%%%%%%%%%%% Linear representation   %%%%%%%%%%%%%%%%%%%%%%%%%%%%%%%%%%%%%%%%%%%%%%%%%%%%%%%
%%%%%%%%%%%%%%%%%%%%%%%%%%%%%%%%%%%%%%%%%%%%%%%%%%%%%%%%%%%%%%%%%%%%%%%%%%%%%%%%%%%%%%%%%%%%%%%%%%%%%%%%%%%%%%%%%%%
%%%%%%%%%%%%%%%%%%%%%%%%%%%%%%%%%%%%%%%%%%%%%%%%%%%%%%%%%%%%%%%%%%%%%%%%%%%%%%%%%%%%%%%%%%%%%%%%%%%%%%%%%%%%%%%%%%%

\subsection{The canonical representation}
\label{subsec-canonical-representation-sln}
\par Let $V = \R^{n+1}$. Let us consider the canonical representation of $G = \SL(n+1, \R)$ on $V$: $g \in G$ acts on $v \in V$ by left matrix multiplication. It induces a canonical representation of $G$ on $\bigwedge^i V$ for every $i=1,2,\dots, n$. For $g \in G$ and 
\[\vect{v} = \vect{v}_1 \wedge \cdots \wedge \vect{v}_i \in \bigwedge\nolimits^i V, \]
$g \vect{v} = (g \vect{v}_1) \wedge \cdots \wedge (g \vect{v}_i).$
\par For $i =1 , \dots, n$, let $\vect{e}_i \in \R^n$ denote the vector with $1$ in the $i$th component and $0$ in other components. 
\par Let us fix a basis for $V$ as follows. Let $\vect{w}_{+} := ( 1, 0, \dots, 0)$. For $i = 1, \dots, n$, let $\vect{w}_i := (0, \dots, 1, \dots, 0)$ with $1$ in the $i+1$st component and $0$ in other components. Then $\{\vect{w}_+, \vect{w}_1, \dots, \vect{w}_n\}$ is a basis for $V$. Let $W$ denote the subspace of $V$ spanned by $\{\vect{w}_1, \dots, \vect{w}_n\}$. For $j= 2, \dots, n$, let $W_j$ the subspace of $W$ spanned by $\{\vect{w}_j, \dots, \vect{w}_n\}$.
\par Let us define 
\begin{equation}
\label{equ:def-Z}
Z : = \left\{ z(\mathfrak{k}) := \begin{bmatrix}1 & \\ & \mathfrak{k}\end{bmatrix}: \mathfrak{k} \in \SO(n) \right\}.\end{equation}
Let us consider the canonical action of $\SO(n)$ on $\R^n$. For $\mathfrak{k} \in \SO(n)$ and $\vect{x} \in \R^n$, let us denote by $\mathfrak{k} \cdot \vect{x}$ the canonical action of $\mathfrak{k}$ on $\vect{x}$. It is straightforward to check that for $ \mathfrak{k} \in \SO(n)$ and $\x \in \R^n$, 
\[z(\mathfrak{k}) U(\x) z\inv(\mathfrak{k}) = U(\mathfrak{k}\cdot \x).\]
\par For any $\x \in \R^n$, let us define a subgroup $\SL(2, \x)$ of $G$ containing $U(\x)$ as follows. For $\x = \vect{e}_1$, let us define 
\[ \SL(2, \vect{e}_1) := \left\{ \begin{bmatrix} h & \\ & \I_{n-1} \end{bmatrix}: h \in \SL(2, \R) \right\}. \] 
For general $\x \in \R^n$, let us choose $\fk \in \SO(n)$ such that $\|\x\|_2\fk \cdot \vect{e}_1 = \x$ and define 
\[ \SL(2, \x) := z(\fk) \SL(2, \vect{e}_1) z\inv(\fk).  \]
It is easy to see that $\SL(2, \x)$ is isomorphic to $\SL(2, \R)$ and $U(\x) \in \SL(2, \x)$ corresponds to \[\begin{bmatrix}1 & \|\x\|_2 \\ & 1\end{bmatrix} \in \SL(2,\R).\] For $r >0$, let $\xi_{\vect{e}_1}(r) \in \SL(2, \vect{e}_1)$ denote the element
\[ \begin{bmatrix} r & 0 & \\ 0 & r\inv & \\ & & \I_{n-1} \end{bmatrix}\]
and $\xi_{\x} (r) \in \SL(2 , \x)$ denote $ z(\fk) \xi_{\vect{e}_1}(r) z\inv(\fk)$. Then $\xi_{\x}(r)$ corresponds to $\begin{bmatrix} r & \\ & r^{-1} \end{bmatrix}$ in $\SL(2, \R)$.
\par Let us study the action of $\SL(2, \x)$ on $V$.
\par Let us first consider the case $\x = \vect{e}_1$. For $r \in \R$, let us denote 
\[ u_1 (r) := U(r \vect{e}_1) ,\]
and
\[U_1 := \{ u_1 (r ): r \in \R \}.\]
Let us denote
\[\Xi_1 : = \{ \xi_1 (r) : = \diag \{r , r\inv,1, \dots,  1 \}: r > 0 \}.\]
It is easy to see that $\xi_1 (r) \vect{w}_+ = r \vect{w}_+$, $u_1 (r) \vect{w}_+ = \vect{w}_+$, $\xi_1 (r) \vect{w}_1 = r\inv \vect{w}_1$, $u_1(r)\vect{w}_1 = \vect{w}_1 + r \vect{w}_+$, and for any $\vect{w} \in W_2$, $\vect{w}$ is fixed by $\SL(2, \vect{e}_1)$.
\par For $\x \in \R^n$, we have $\x = \|\x\|_2 \mathfrak{k} \cdot \vect{e}_1$ for some $\mathfrak{k} \in \SO(n)$ and 
\[\SL(2,\x) = z(\mathfrak{k}) \SL(2, \vect{e}_1) z\inv(\mathfrak{k}). \] 
 In particular, we have that 
\[U(\x) = z(\mathfrak{k}) u_1(\|\x\|_2)z\inv(\mathfrak{k})\]
 and $\xi_{\x}(r) = z(\mathfrak{k}) \xi_1(r) z\inv(\mathfrak{k})$. Since $z(\mathfrak{k}) \vect{w}_+ = \vect{w}_+$ and $z(\mathfrak{k}) W = W$, we have that $\xi_{\x}(r) \vect{w}_+ = r \vect{w}_+$, $U(\x) \vect{w}_+ = \vect{w}_+$, $\xi_{\x}(r) z(\mathfrak{k}) \vect{w}_1 = r\inv \mathfrak{k} \cdot \vect{w}_1$, $U(\x) z(\mathfrak{k}) \vect{w}_1 = z(\mathfrak{k}) \vect{w}_1 + \|\x\|_2 \vect{w}_+$ and for any $\vect{w} \in z(\mathfrak{k})W_2$, $\vect{w}$ is fixed by $\SL(2, \x)$.
\par Let us consider the action of $\SL(2,\x)$ on $\bigwedge^i V$ for $i = 2,\dots, n$. Let us denote $\x = \|\x\|_2 \mathfrak{k} \cdot \vect{e}_1$ as above. For any $\vect{w} \in \bigwedge^{i-1} z(\mathfrak{k})W_2$, we have that 
\[\xi_{\x}(r) ((z(\mathfrak{k})\vect{w}_1)\wedge \vect{w}) = r\inv ((z(\mathfrak{k})\vect{w}_1)\wedge \vect{w}),\]
\[U(\x)((z(\mathfrak{k})\vect{w}_1)\wedge \vect{w}) =(z(\mathfrak{k})\vect{w}_1)\wedge \vect{w}  + \|\x\|_2 (\vect{w}_+ \wedge \vect{w}),\]
\[\xi_{\x}(r) (\vect{w}_+ \wedge \vect{w}) = r (\vect{w}_+ \wedge \vect{w})\]
  and 
  \[U(\x)(\vect{w}_+ \wedge \vect{w}) = \vect{w}_+ \wedge \vect{w}.\]
  For any $\vect{w} \in \bigwedge^{i} z(\mathfrak{k})W_2$ and any $\vect{w}' \in \bigwedge^{i-2} z(\mathfrak{k})W_2$, we have that $\vect{w}$ and $\vect{w}_+ \wedge (z(\mathfrak{k})\vect{w}_1) \wedge \vect{w}'$ are fixed by $\SL(2,\x)$.
%%%%%%%%%%%%%%%%%%%%%%%%%%%%%%%%%%%%%%%%%%%%%%%%%%%%%%%%%%%%%%%%%%%%%%%%%%%%%%%%%%%%%%%%%%%%%%%%%%%%%%%%%%%%%%%%%%
%%%%%%%%%%%%%%%%%%%%%%%%%%%%%%%%%%%%%%%%%%%%%%%%%%%%%%%%%%%%%%%%%%%%%%%%%%%%%%%%%%%%%%%%%%%%%%%%%%%%%%%%%%%%%%%%%%
%%%%%%%%%%%%%%%%%%%%%%%%%%%%%%%%%%%%%%%%%%%%  lattices in R^n+1 %%%%%%%%%%%%%%%%%%%%%%%%%%%%%%%%%%%%%%%%%%%%%%%%%%
%%%%%%%%%%%%%%%%%%%%%%%%%%%%%%%%%%%%%%%%%%%%%%%%%%%%%%%%%%%%%%%%%%%%%%%%%%%%%%%%%%%%%%%%%%%%%%%%%%%%%%%%%%%%%%%%%%
%%%%%%%%%%%%%%%%%%%%%%%%%%%%%%%%%%%%%%%%%%%%%%%%%%%%%%%%%%%%%%%%%%%%%%%%%%%%%%%%%%%%%%%%%%%%%%%%%%%%%%%%%%%%%%%%%%

 \subsection{Lattices in $\R^{n+1}$}
 \label{subsec-lattices-in-Rn+1}
 \par In this subsection let us recall some basic facts on lattices and sublattices in $\R^{n+1}$. 
 \par For a discrete subgroup $\Delta$ of $\R^{n+1}$, let $\Span_{\R}(\Delta)$ denote the $\R$-span of $\Delta$.
 \par Let $\Lambda \in X = G/\Gamma$ be a unimodular lattice in $\R^{n+1}$. For $i=1, \dots, n+1$, let $\cL_i( \Lambda)$ denote the collection of $i$-dimensional sublattices of $\Lambda$. Given $\Lambda' \in \cL_i( \Lambda)$, let us choose a basis $\{\vect{v}_1, \dots, \vect{v}_i\}$ of $\Lambda'$ and define 
 \begin{equation} 
 \label{equ:lattice-to-wedge}
 \cW(\Lambda') := \vect{v}_1 \wedge \cdots \wedge \vect{v}_i \in \bigwedge^i V.
 \end{equation}
 $\cW(\Lambda')$ is well defined modulo $\pm 1$. Thus $\cW$ defines a map from $\cL_i( \Lambda)$ to $\bigwedge^i V/\pm $ for each $i =1, \dots, n+1$. Let us denote $d(\Lambda') := \|\cW(\Lambda')\|$. We say that $\Lambda'$ is primitive relative to $\Lambda$ if $\cW(\Lambda')$ can not be written as $m \cW(\tilde{\Lambda})$ where $|m| >1$ is an integer and $\tilde{\Lambda} \in \cL_i(\Lambda)$ (see \cite{Cassels}).
\par For $ j = 1, \dots, i$, let 
\[\lambda_j(\Lambda') := \inf\{r\geq 0: B(\vect{0}, r) \text{ contains at least } j \text{ linearly independent vectors of } \Lambda'\}.\]
By the Minkowski Theorem (see \cite{Cassels}), we have the following:
\begin{equation}
\label{equ:minkowski-thm}
 \lambda_1(\Lambda')\cdots \lambda_i(\Lambda') \asymp d(\Lambda').
\end{equation}
Moreover, there exists a basis (called Minkowski reduced basis) of $\Lambda'$, $\{\vect{v}_j: j = 1, \dots, i\}$, such that $\|\vect{v}_j\| \asymp \lambda_j(\Lambda')$ for every $j = 1, \dots, i$.
\par For $\rho >0$ and $i =1, \dots, n+1$, let $\mathcal{C}_i(\Lambda, \rho)$ denote the collection of $i$-dimensional primitive sublattices $\Lambda'$ of $\Lambda$ with $d(\Lambda') < \rho$. We will need the following result on counting sublattices: 
\begin{proposition}
\label{prop:counting-sublattices}
\par There exists a constant $N >1$ such that the following statement holds. For any $0 <\epsilon < 1 $ and any $i= 1, \dots, n$, let $\Lambda \in K_{\epsilon}$ where $K_{\epsilon}$ is defined in \eqref{equ:compact-subset-homo-space}.  Then we have that
\[\sharp\mathcal{C}_i(\Lambda, 1) \leq \epsilon^{-N} .\]
\end{proposition}
\begin{proof}
\par First note that there exists a constant $N_1 >1$ such that for any $i=1,\dots, n$ and $\rho >0$, 
\[ \sharp\mathcal{C}_i(\Z^{n+1}, \rho) \leq \rho^{N_1}.\]
\par We also note that there exists a constant $N_2 >1$ such that for any $\Lambda \in K_{\epsilon}$, there exists $g \in \SL(n+1, \R)$ with $\|g\inv\| < \epsilon^{-N_2}$ such that $\Lambda = g \Z^{n+1}$. In fact, the fact is easily seen if $g$ is chosen in a Siegel set (see \cite[Proposition 10.56]{EW2017}). Let us fix $\rho > \epsilon$ and $i = 1, \dots, n$. Then for any $\Lambda' \in \mathcal{C}_i(\Lambda, 1)$, then we have that $g\inv \Lambda' \subset \Z^{n+1}$ and 
\[d(g^{-1} \Lambda') \leq \|g\inv\|^i d(\Lambda') \leq \epsilon^{-(n+1) N_2} .\]
Therefore, we have that 
\[ \sharp\mathcal{C}_i(\Lambda, 1) \leq \sharp\mathcal{C}_i(\Z^{n+1}, \epsilon^{-(n+1) N_2}) \leq \epsilon^{-N} \]
where $N = N_1 N_2 (n+1)$. 
\par This completes the proof.
\end{proof}

%%%%%%%%%%%%%%%%%%%%%%%%%%%%%%%%%%%%%%%%%%%%%%%%%%%%%%%%%%%%%%%%%%%%%%%%%%%%%%%%%%%%%%%%%%%%%%%%%%%%%%%%%%%%%%%%%%%
%%%%%%%%%%%%%%%%%%%%%%%%%%%%%%%%%%%%%%%%%%%%%%%%%%%%%%%%%%%%%%%%%%%%%%%%%%%%%%%%%%%%%%%%%%%%%%%%%%%%%%%%%%%%%%%%%%%
%%%%%%%%%%%%%%%%%%%%%%%%%%%%%%%%%%%%%%% M-Cantor Rich construction  %%%%%%%%%%%%%%%%%%%%%%%%%%%%%%%%%%%%%%%%%%%%%%%
%%%%%%%%%%%%%%%%%%%%%%%%%%%%%%%%%%%%%%%%%%%%%%%%%%%%%%%%%%%%%%%%%%%%%%%%%%%%%%%%%%%%%%%%%%%%%%%%%%%%%%%%%%%%%%%%%%%
%%%%%%%%%%%%%%%%%%%%%%%%%%%%%%%%%%%%%%%%%%%%%%%%%%%%%%%%%%%%%%%%%%%%%%%%%%%%%%%%%%%%%%%%%%%%%%%%%%%%%%%%%%%%%%%%%%%
\section{A Cantor like construction}
\label{sec:cantor-like-construction}
\par In this section, we will introduce a Cantor like construction which will help us to compute Hausdorff dimension.
\par Since we focus on the case of curves, we may assume that $\mathcal{U}$ is given by 
\[ \bphi = (\varphi_1 , \dots, \varphi_n): [0,1] \to \R^n \]
 where every $\varphi_i(s)$ is a $C^n$ differentiable function. 

\begin{definition}[See~{\cite[\S 5]{Beresnevich}}]
\label{def:cantor-like-construction}
\par For an integer $R >0$ and a closed interval $J \subset [0,1]$, let us denote by $\Par_{R}(J)$ the collection of closed intervals obtained by dividing $J$ into $R$ closed intervals of the same size. For a collection $\mathcal{I}$ of closed intervals, let us denote
\[\Par_{R}(\mathcal{I}) := \bigcup_{I \in \mathcal{I}} \Par_R(I).\] 
A sequence $\{\mathcal{I}_q\}_{q \in \N}$ of collections of closed intevals is called a $R$-sequence if for every $q \geq 1$, $\mathcal{I}_q \subset \Par_R(\mathcal{I}_{q-1})$. For a $R$-sequence $\{\mathcal{I}_q\}_{q \in \N}$ and $q \geq 1$, let us define $\hat{\mathcal{I}}_q := \Par_R(\mathcal{I}_{q-1})\setminus \mathcal{I}_q$ and 
\[\mathcal{K}(\{\mathcal{I}_q: q \in \N\}) := \bigcap_{q \in \N} \bigcup_{I_q \in \mathcal{I}_q} I_q .\]
Then every $R$-sequence $\{\mathcal{I}_q\}_{q \in \N}$ gives a Cantor like subset $\mathcal{K}(\{\mathcal{I}_q\}_{q \in \N})$ of $[0,1]$.
\par For $q \geq 1$ and a partition $\{\hat{\mathcal{I}}_{q, p}\}_{0 \leq p \leq q-1}$ of $\hat{\mathcal{I}}_q$, let us define 
\[d_q(\{\hat{\mathcal{I}}_{q, p}\}_{0\leq p \leq q-1}):= \sum_{p=0}^{q-1} \left( \frac{4}{R}\right)^{q-p} \max_{I_p \in \mathcal{I}_p} F( \hat{\mathcal{I}}_{q,p}, I_p), \]
where $F( \hat{\mathcal{I}}_{q,p}, I_p) := \sharp\{I_q \in \hat{\mathcal{I}}_{q,p}, I_q \in I_p\}$. Let us define 
\[d_q(\mathcal{I}_q) := \min_{\{\hat{\mathcal{I}}_{q, p}\}_{0\leq p \leq q-1}} d_q(\{\hat{\mathcal{I}}_{q, p}\}_{0\leq p \leq q-1}),\]
where $\{\hat{\mathcal{I}}_{q, p}\}_{0\leq p \leq q-1}$ runs over all possible partitions of $\hat{\mathcal{I}}_q$. Let us define 
\[d(\{\mathcal{I}_q\}_{q \in \N}) := \max_{q \in \N} d_q(\mathcal{I}_q).\]
\end{definition}

\begin{definition}[See~{\cite[\S 5]{Beresnevich}}]
\label{def:m-rich}
For $R >1$ and a compact subset $X \subset [0,1]$, we say that $X$ is $R$-Cantor rich if for any $\epsilon >0$, there exists a $R$-sequence $\{\mathcal{I}_q\}_{q \in \N}$ such that 
\[\mathcal{K}(\{\mathcal{I}_q\}_{q \in \N}) \subset X\]
and $d(\{\mathcal{I}_q\}_{q \in \N}) \leq \epsilon$.
\end{definition}
\par Our proof relies on the following two theorems:
\begin{theorem}[See~{\cite[Theorem 6]{Beresnevich}}]
\label{thm:hausdorff-dimension-m-rich}
Any $R$-Cantor rich set $X$ has full Hausdorff dimension.
\end{theorem}

\begin{theorem}[See~{\cite[Theorem 7]{Beresnevich}}]
\label{thm:intersection-m-rich}
Any countable intersection of $R$-Cantor rich sets in $[0,1]$ is $R$-Cantor rich.
\end{theorem}

\par To show Theorem \ref{thm:main-thm} and \ref{thm:double-intersection}, it suffices to find a constant $R >1$ and show that for any weight $\vect{r}$, $\bphi\inv(\Bad(\vect{r})\cap \bphi([0,1]))$ is $R$-Cantor rich. We will determine $R>1$ later.
\begin{theorem}
\label{thm:main-task}
There exists a constant $R > 1$ such that for any weight $\vect{r}$, $\bphi\inv(\Bad(\vect{r})\cap \bphi([0,1]))$ is $R$-Cantor rich.
\end{theorem} 
\par Our main task is to prove Theorem \ref{thm:main-task}.
\par Let us fix $R$. We will show that for any $\epsilon >0$, we can construct a $R$-sequence $\{\mathcal{I}_q\}_{q \in \N}$ such that $\mathcal{K}(\{\mathcal{I}_q\}_{q \in \N}) \subset \bphi\inv(\Bad(\vect{r}))$ and $d(\{\mathcal{I}_q\}_{q \in \N}) < \epsilon$. 
%%%%%%%%%%%%%%%%%%%%%%%%%%%%%%%%%%%%%%%%%%%%%%%%%%%%%%%%%%%%%%%%%%%%%%%%%%%%%%%%%%%%%%%%%%%%%%%%%%%%%%%%%%%%%%%%
%%%%%%%%%%%%%%%%%%%%%%%%%%%%%%%%%%%%%%%%%%%%%%%%%%%%%%%%%%%%%%%%%%%%%%%%%%%%%%%%%%%%%%%%%%%%%%%%%%%%%%%%%%%%%%%%
%%%%%%%%%%%%%%%%%%%%%%%%%%%%%%%%%%%%%%%%%  Standing Assumptions    %%%%%%%%%%%%%%%%%%%%%%%%%%%%%%%%%%%%%%%%%%%%%
%%%%%%%%%%%%%%%%%%%%%%%%%%%%%%%%%%%%%%%%%%%%%%%%%%%%%%%%%%%%%%%%%%%%%%%%%%%%%%%%%%%%%%%%%%%%%%%%%%%%%%%%%%%%%%%%
%%%%%%%%%%%%%%%%%%%%%%%%%%%%%%%%%%%%%%%%%%%%%%%%%%%%%%%%%%%%%%%%%%%%%%%%%%%%%%%%%%%%%%%%%%%%%%%%%%%%%%%%%%%%%%%%
\begin{StandingAssumption}
\label{standing-assumptions}
\par Let us make some assumptions to simplify the proof.
\begin{enumerate}[label=\textbf{A.\arabic*}]
\item \label{assumption-1} Without loss of generality, we may assume that $r_1 \geq r_2 \geq \cdots \geq r_n$. We may also assume that $r_n >0$. By \cite{Beresnevich}, if $r_n = 0$, we can reduce the problem to the $n-1$ dimensional case.
\item \label{assumption-2} Since 
\[\bphi =(\varphi_1, \dots, \varphi_n): [0,1] \to \R^{n}\]
 is $C^n$ differentiable and non-degenerate, we may assume that for any $s \in [0,1]$ and any $i = 1 ,\dots, n$, $\varphi'_i(s) \neq 0$. If this is not the case, we can replace $[0,1]$ with a smaller closed interval $I \subset [0,1]$, cf. \cite[Property F]{Beresnevich}. Then since $[0,1]$ is closed, there exist constants $C_1 > c_1 >0$ such that for any $s \in [0,1]$ and any $i=1,\dots, n$, $c_1 \leq |\varphi'_i(s)| \leq C_1$. 
\end{enumerate}
 \end{StandingAssumption}
%%%%%%%%%%%%%%%%%%%%%%%%%%%%%%%%%%%%%%%%%%%%%%%%%%%%%%%%%%%%%%%%%%%%%%%%%%%%%%%%%%%%%%%%%%%%%%%%%%%%%%%%%%%%%
%%%%%%%%%%%%%%%%%%%%%%%%%%%%%%%%%%%%%%%%%%%%%%%%%%%%%%%%%%%%%%%%%%%%%%%%%%%%%%%%%%%%%%%%%%%%%%%%%%%%%%%%%%%%%
%%%%%%%%%%%%%%%%%%%%%%%%%%%%%%%%%%%%% End Standing Assumptions  %%%%%%%%%%%%%%%%%%%%%%%%%%%%%%%%%%%%%%%%%%%%%
%%%%%%%%%%%%%%%%%%%%%%%%%%%%%%%%%%%%%%%%%%%%%%%%%%%%%%%%%%%%%%%%%%%%%%%%%%%%%%%%%%%%%%%%%%%%%%%%%%%%%%%%%%%%%
%%%%%%%%%%%%%%%%%%%%%%%%%%%%%%%%%%%%%%%%%%%%%%%%%%%%%%%%%%%%%%%%%%%%%%%%%%%%%%%%%%%%%%%%%%%%%%%%%%%%%%%%%%%%%
\par Let us fix some notation. Let $ \kappa >0$ be a small parameter which we will determine later. Let $b >0$ be such that $b^{1+r_1}=R$. For $t >0$, let us denote 
\[ g_{\vect{r}} (t) := \begin{bmatrix} b^{t} & & & \\ & b^{-r_1 t} & & \\ & & \ddots & \\ & & & b^{-r_n t} \end{bmatrix}.\] 
For $i =1 , \dots, n$, let $\lambda_i = \frac{1+r_i}{1+r_1}$. Then we have that $1 = \lambda_1 \geq \lambda_2 \geq \cdots \geq \lambda_n$. Let $m(\cdot)$ denote the Lebesgue measure on $[0,1]$. 

\par Let us give the $R$-sequence as follows. Let $\mathcal{I}_0 = \{ [0,1] \}$. Suppose that we have defined $\mathcal{I}_{q-1}$ for $q \geq 1$ and every $I_{q-1} \in \mathcal{I}_{q-1}$ is a closed interval of size $ R^{-q+1}$. Let us define $\mathcal{I}_{q}\subset \Par_R(\mathcal{I}_{q-1})$ as follows. For any $I_{q} \in \Par_R(\mathcal{I}_q)$, $I_{q} \in \hat{\mathcal{I}}_{q}$ if and only if there exists $s \in I_{q}$ such that $g_{\vect{r}} (q) U(\bphi(s)) \Z^{n+1} \notin K_{\kappa}$. That is to say, there exists $\vect{a} \in \Z^{n+1}\setminus \{\vect{0}\}$ such that $\|g_{\vect{r}} (q) U(\bphi(s))\vect{a}\| \leq \kappa$. Let us define $\mathcal{I}_q = \Par_R(\mathcal{I}_{q-1})\setminus \hat{\mathcal{I}}_q$. This finishes the construction of $\{\mathcal{I}_q\}_{q\in \N}$. It is easy to see that 
\[ \mathcal{K}(\{\mathcal{I}_q\}_{q \in \N}) \subset \bphi\inv(\Bad(\vect{r})). \]
\par We need to prove the following:
\begin{proposition}
\label{prop:distance-sequence-small}
For any $\epsilon >0$, there exists $\kappa >0$ such that the $R$-sequence $\{\mathcal{I}_q\}_{q \in \N}$ constructed as above with $\kappa $ satisfies that
\begin{equation}
\label{equ:distance-sequence-small}
d(\{\mathcal{I}_q\}_{q\in \N}) \leq \epsilon.
\end{equation}
\end{proposition}
\par Let $N >1$ be the constant from Proposition \ref{prop:counting-sublattices} and $k>0$ be such that $\kappa = R^{-k}$. We can choose $\kappa$ so that $k$ is an integer. Let us give a partition $\{\hat{\mathcal{I}}_{q,p}\}_{0\leq p \leq q-1}$ of $\hat{\mathcal{I}}_q$ for each $q \in \N$ which shows that Proposition \ref{prop:distance-sequence-small} holds.
\begin{definition}
\label{def:finer-partition-Iq}
\par Let us fix a small constant $0 < \rho < 1 $. We will modify the choice of $\rho$ later in this paper according to the constants arising from our technical results. For $q \leq 10^6 n^4 N k$, let us define $\hat{\mathcal{I}}_{q, 0} := \hat{\mathcal{I}}_q$ and $\hat{\mathcal{I}}_{q,p} = \emptyset$ for other $p$'s.
\par For $q > 10^6 n^4 N k$ and $l = 2000 n^2 N k$, let $ p = q - 2l$. Let us define $\hat{\mathcal{I}}_{q,p'} := \emptyset$ for $p<p' \leq q-1$. Let us define $\hat{\mathcal{I}}_{q,p}$ to be the collection of $I_q \in \hat{\mathcal{I}}_q $ with the following property: there exists $s \in I_q$ such that for any $j=1,\dots, n$ and any $\vect{w} = \vect{w}_1 \wedge \cdots \wedge \vect{w}_j \in \bigwedge^j \Z^{n+1} \setminus \{\vect{0}\}$, 
\[ \max \{\|g_{\vect{r}}(q) U(\bphi(s')) \vect{w}\| : s' \in  [s-R^{-q+l}, s+R^{-q+l}] \} \geq \rho^j.\]
\par Let $\eta = \frac{1}{100n^2}$ and $\eta' = \frac{\eta}{1+r_1}$. For $q > 10^6 n^4 N  k$ and $ 2000  n^2 N k < l \leq 2\eta' q$, let $p = q - 2l$. Let us define $\hat{\mathcal{I}}_{q,p+1} := \emptyset$. For $j = 1, \dots, n$, let us define $\hat{\mathcal{I}}_{q,p}(j)$ to be the collection of $I_q \in \hat{\mathcal{I}}_q\setminus \left( \bigcup_{p' < p} \hat{\mathcal{I}}_{q,p'} \right)$ such that there exists $s  \in I_q$ and $\vect{v} = \vect{v}_1 \wedge \cdots \wedge \vect{v}_j \in \bigwedge^j \Z^{n+1} \setminus\{\vect{0}\}$ such that
\[ \|g_{\vect{r}}(q) U(\bphi(s'))\vect{v}\| < \rho^j,\]
for any $s' \in [s - R^{-q + l}, s + R^{-q +l}]$ and for any $j' = 1, \dots, n$ and any $\vect{w} = \vect{w}_1 \wedge \cdots \wedge \vect{w}_{j'} \in \bigwedge^{j'} \Z^{n+1} \setminus \{\vect{0}\}$,
\[\max\left\{ \|g_{\vect{r}}(q) U(\bphi(s')) \vect{w}\|:s' \in  [s - R^{-q+l+1}, s + R^{-q +l+1} ]\right\} \geq \rho^{j'}.\]
Let us define $\hat{\mathcal{I}}_{q,p} = \bigcup_{j=1}^{n} \hat{\mathcal{I}}_{q,p}(j)$.
 \par For $j = 1, \dots, n$, let us define $\hat{\mathcal{I}}_{q,0}(j)$ to be the collection of $I_q \in \hat{\mathcal{I}}_q\setminus \left( \bigcup_{p' \leq q- 4\eta' q} \hat{\mathcal{I}}_{q,p'} \right)$ such that there exists $s \in I_q$ and $\vect{v}= \vect{v}_1 \wedge \cdots \wedge \vect{v}_j \in \bigwedge^j \Z^{n+1}\setminus\{\vect{0}\}$ such that
\[\max\left\{ \|g_{\vect{r}}(q) U(\bphi(s')) \vect{v}\|:s' \in [s - R^{-q (1-2\eta')}, s+ R^{-q(1-2\eta')} ]\right\} <  \rho^j.\]
Let us define $\hat{\mathcal{I}}_{q,0} = \bigcup_{j=1}^{n} \hat{\mathcal{I}}_{q,0}(j)$.
\par Let us define $\hat{\mathcal{I}}_{q,p} := \emptyset$ for other $p$'s. It is easy to see that $\{\hat{\mathcal{I}}_{q,p}\}_{0 \leq p \leq q-1}$ is a partition of $\hat{\mathcal{I}}_q$.
\end{definition}
\par Besides the definition of $\{\hat{\mathcal{I}}_{q,p}\}_{ 0 \leq p \leq q-1 }$, let us also introduce the notion of dangerous interval and extremely dangerous interval:
\begin{definition}
\label{def:dangerous-interval}
For $q > 10^6 n^4 N k$, $ 1000 n^2 N k \leq l \leq \eta' q $, and $\vect{a} \in \Z^{n+1}\setminus \{\vect{0}\}$, the $(q,l)$-dangerous interval associated with $\vect{a}$, which is denoted by $\Delta_{q,l} (\vect{a})$, is a closed interval of the form $\Delta_{q,l}(\vect{a}) =  [ s - R^{-q + l }, s + R^{-q +l}] \subset [0,1]$ such that $I_q \subset \Delta_{q,l}(\vect{a})$ for some $I_q \in \hat{\mathcal{I}}_q$,
\[ \max  \{ \|g_{\vect{r}}(q) U(\bphi(s')) \vect{a}\|: s' \in \Delta_{q,l}(\vect{a}) \} < \rho \]
and
\[\max \{ \|g_{\vect{r}}(q) U(\bphi(s')) \vect{a}\|: s' \in  [s - R^{-q+l+1}, s + R^{-q+l+1}] \} \geq \rho.\]
 The center $s$ of $\Delta_{q,l}(\vect{a})$ is chosen such that the first coordinate of $U(\bphi(s)) \vect{a}$ is zero.
\par For $q \geq 10^6 n^4 N k $ and $\vect{a} \in \Z^{n+1}\setminus \{\vect{0}\}$, the $q$-extremely dangerous interval associated with $\vect{a}$, which is denoted by $\Delta_q (\vect{a})$, is a closed interval of the form $\Delta_q(\vect{a}) =  [s - R^{-q + l'}, s + R^{-q + l'}]$ with $l' > \eta' q$ such that $I_q \subset \Delta_q(\vect{a})$ for some $I_q \in \hat{\mathcal{I}}_q$,
\[ \max  \{ \|g_{\vect{r}}(q) U(\bphi(s')) \vect{a}\|: s' \in \Delta_{q}(\vect{a}) =  [s - R^{-q+l'}, s + R^{-q+l'}]\} < \rho \]
and
\[\max \{ \|g_{\vect{r}}(q) U(\bphi(s')) \vect{a}\|: s' \in  [s - R^{-q+l'+1}, s + R^{-q+l'+1}] \} \geq \rho.\]
\end{definition}
\begin{remark}
Note that for any $q \geq 10^6 n^4 N k$, there are only finitely many $\vect{a}$'s such that $\Delta_{q, l} (\vect{a})$ or $\Delta_q (\vect{a})$ exist.
\end{remark}

%%%%%%%%%%%%%%%%%%%%%%%%%%%%%%%%%%%%%%%%%%%%%%%%%%%%%%%%%%%%%%%%%%%%%%%%%%%%%%%%%%%%%%%%%%%%%%%%%%%%%%%%%
%%%%%%%%%%%%%%%%%%%%%%%%%%%%%%%%%%%%%%%%%%%%%%%%%%%%%%%%%%%%%%%%%%%%%%%%%%%%%%%%%%%%%%%%%%%%%%%%%%%%%%%%%
%%%%%%%%%%%%%%%%%%%%%%%%%%%%%%%%%%%%% Couniting dangerous interval %%%%%%%%%%%%%%%%%%%%%%%%%%%%%%%%%%%%%%
%%%%%%%%%%%%%%%%%%%%%%%%%%%%%%%%%%%%%%%%%%%%%%%%%%%%%%%%%%%%%%%%%%%%%%%%%%%%%%%%%%%%%%%%%%%%%%%%%%%%%%%%%
%%%%%%%%%%%%%%%%%%%%%%%%%%%%%%%%%%%%%%%%%%%%%%%%%%%%%%%%%%%%%%%%%%%%%%%%%%%%%%%%%%%%%%%%%%%%%%%%%%%%%%%%%
\section{Counting dangerous intervals}
\label{sec:count-dangerous-intervals}
\par In this section we will count dangerous intervals and extremely dangerous intervals.
\begin{proposition}
\label{prop:count-dangerous-intervals}
Let $q \geq 10^6 n^4 N k$, $ 1000 n^2 N k  \leq l \leq \eta' q$ and $p = q - 2l $. For $I_p \in \mathcal{I}_p$, let 
$\mathcal{D}_{q, l} (I_p)$ denote the collection of $(q,l)$-dangerous intervals which intersect $I_p$. Then for any $I_p \in \mathcal{I}_p$, 
\[\sharp\mathcal{D}_{q,l}(I_p) \ll R^{(1- \frac{1}{10n}) l} .\]
\end{proposition}
\begin{proposition}
\label{prop:count-extremely-dangerous-intervals}
Let $q \geq 10^6 n^4 N k$. Let $D_q \subset [0,1]$ denote the union of $q$-extremely dangerous intervals contained in $[0,1]$. Then
$D_q$ can be covered by a collection of $N_q$ closed intervals of length $\delta_q$ and 
\[ N_q \leq \frac{K_0 (\rho^{n+1} b^{-\eta q})^{\alpha}}{\delta_q} \]
where $\delta_q =  R^{-q(1-\eta')}$, $K_0>0$ is a constant, and $\alpha = \frac{1}{(n+1)(2n-1)}$.
\end{proposition}
\par In fact, Proposition \ref{prop:count-extremely-dangerous-intervals} is a rephrase of the following theorem due to Bernik, Kleinbock and Margulis:
\begin{theorem}[See~{\cite[Proposition 2]{Beresnevich}} and~{\cite[Theorem 1.4]{Bernik-Kleinbock-Margulis}}]
\label{thm:bernik-kleinbock-margulis}
Let $q > 10^6 n^4 N k$. Let us define $E_q \subset [0,1]$ to be the set of $s \in [0,1]$ such that there exists $\vect{a} =(a_0, a_1, \dots, a_n) \in \Z^{n+1}\setminus\{\vect{0}\}$ such that $|a_i| < \rho b^{r_i q}$ for $i=1,\dots,n$, $|f(s)| < \rho b^{-q}$ and $| f'(s)| <  b^{(r_1 - \eta)q}$ where 
\begin{equation}
\label{equ:def-f}
f(s) = a_0 + a_1 \varphi_1(s) + \cdots + a_n \varphi_n(s).
\end{equation}
Then $E_q$ can be covered by a collection $\mathcal{E}_{q}$ of intervals such that
\[m(\Delta) \leq \delta_q \text{ for all } \Delta \in \mathcal{E}_{q},\]
and
\[|\mathcal{E}_{q}| \leq \frac{K_0 (\rho^{n+1} b^{-\eta q})^{\alpha}}{\delta_q},\]
where $\delta_q =  R^{-q(1 - \eta')}$, $K_0 >0$ is a constant, and $\alpha = \frac{1}{(n+1)(2n-1)}$.
\end{theorem}
\par The theorem above is a simplified version of \cite[Theorem 1.4]{Bernik-Kleinbock-Margulis}. The original version is more general.
\begin{proof}[Proof of Proposition \ref{prop:count-extremely-dangerous-intervals}]
\par For every $q$-extremely dangerous interval $\Delta_q(\vect{a}) =  [s - R^{-q+l'}, s + R^{-q +l'}]$ where $l' \geq \eta' q$ and $\vect{a} = (a_0, a_1, \dots, a_n)$, we have that 
\begin{equation}
\label{equ:small-vector}
\|g_{\vect{r}}(q)U(\bphi(s'))\vect{a}\| < \rho
\end{equation}
for every $s' \in \Delta_q (\vect{a})$. By direct computation, we have that 
\[ g_{\vect{r}}(q) U(\bphi(s'))\vect{a} = (v_0(s'), v_1(s'), \dots, v_n(s')) \]
where 
\[v_0(s') = b^q (a_0 + a_1 \varphi_1 (s') + \cdots + a_n \varphi_n (s')),\]
 and $v_i (s') = b^{-r_i q} a_i$ for $i =1, \dots, n$. 
 Then \eqref{equ:small-vector} implies that $|a_i| < \rho b^{r_i q}$ for $i=1, \dots, n$, and $|f(s)| < \rho b^{-q}$, where $f$ is as in \eqref{equ:def-f}. Since $l \ge \eta' q$, we have that 
 \[|f(s')| < \rho b^{-q} \]
 for any $s' \in  [s - R^{-q(1-\eta')} , s + R^{-q(1-\eta')}]$. Let us write $s' = s + r R^{-q(1-\eta')}$ for some $r \in [ -1, 1]$. Then 
 \[
f(s')  =  f(s) +   f'(s)r R^{-q(1-\eta')} + O(R^{-2q(1-\eta')}).
 \]
 Therefore, we have that for any $r \in [ -1,1]$, 
 \begin{align*}
 | f'(s) r R^{-q(1-\eta')}| &=  |f(s') - f(s) - O(R^{-2q(1-\eta')})| \\ 
                            &\leq  |f(s')| + |f(s)| + O(R^{-2q(1-\eta')}) \\
                            &<  \rho b^{-q} + \rho b^{-q} + \rho b^{-q} < b^{-q}.
 \end{align*}
 This implies that 
 \[
 | f'(s)| < R^{q(1-\eta')} b^{-q} = b^{q(r_1-\eta)}.
 \]
The last equality above holds because $b^{1+r_1} = R$ and $\eta' = \frac{\eta}{1+r_1}$. This shows that 
$x \in E_q$ for any $x \in \Delta_q(\vect{a})$, i.e., $\Delta_q(\vect{a}) \subset E_q$. Therefore, we have that 
$D_q \subset E_q$. Then the conclusion follows from Theorem \ref{thm:bernik-kleinbock-margulis}.
 \end{proof}
%%%%%%%%%%%%%%%%%%%%%%%%%%%%%%%%%%%%%%%%%%%%%%%%%%%%%%%%%%%%%%%%%%%%%%%%%%%%%%%%%%%%%%%%%%%%%%%%%%%%%%%%%%%%%%%%%%%%%
%%%%%%%%%%%%%%%%%%%%%%%%%%%%%%%%%%%%%%%%%%%%%%%%%%%%%%%%%%%%%%%%%%%%%%%%%%%%%%%%%%%%%%%%%%%%%%%%%%%%%%%%%%%%%%%%%%%%%
%%%%%%%%%%%%%%%%%%%%%%%%%%%%%%%%%%%%%%% counting dangerous interval %%%%%%%%%%%%%%%%%%%%%%%%%%%%%%%%%%%%%%%%%%%%%%%%%
%%%%%%%%%%%%%%%%%%%%%%%%%%%%%%%%%%%%%%%%%%%%%%%%%%%%%%%%%%%%%%%%%%%%%%%%%%%%%%%%%%%%%%%%%%%%%%%%%%%%%%%%%%%%%%%%%%%%%
%%%%%%%%%%%%%%%%%%%%%%%%%%%%%%%%%%%%%%%%%%%%%%%%%%%%%%%%%%%%%%%%%%%%%%%%%%%%%%%%%%%%%%%%%%%%%%%%%%%%%%%%%%%%%%%%%%%%%
%%%%%%%%%%%%%%%%%%%%%%%%%%%%%%%%%%%%%%%%%%%%%%%%%%%%%%%%%%%%%%%%%%%%%%%%%%%%%%%%%%%%%%%%%%%%%%%%%%%%%%%%%%%%%%%%%%%%%

\par The rest of the section is devoted to the proof of Proposition \ref{prop:count-dangerous-intervals}. This is one of the main technical results of this paper.
\begin{proof}[Proof of Proposition \ref{prop:count-dangerous-intervals}]
\par Let us fix $I_p \in \mathcal{I}_p$. Let us write $I_p =  [s - R^{-q + 2l}, s+ R^{-q + 2l}]$. We claim that we can approximate $\bphi(I_p)$ by its linear part. In fact, for any $s' \in I_p$, let us write $s' = s + r R^{-q + 2l}$ for some $r \in [-1, 1]$. By Taylor's expansion, we have that
\begin{align*}
g_{\vect{r}}(q)U(\bphi(s')) &=  g_{\vect{r}}(q) U(\bphi(s) +  R^{-q + 2l} r\bphi'(s) + O(R^{-2q + 4l})) \\ 
&= g_{\vect{r}}(q) U(O(R^{-2q + 4l})) g_{\vect{r}}(-q) g_{\vect{r}}(q) U(\bphi(s) +  R^{-q + 2l} r\bphi'(s) ) \\ 
&= U(O(R^{-q + 4l})) g_{\vect{r}}(q) U(\bphi(s) +  R^{-q + 2l}r \bphi'(s)).
\end{align*}
Since $l \leq \eta' q$, we have that $O(R^{-q + 4l})$ is exponentially small and thus can be ignored. Therefore, we can approximate $\bphi(s')$ by $\bphi(s) +  \bphi'(s)(s'-s)$ for any $s' \in I_p$.

\par Let us take a $(q,l)$-dangerous interval $\Delta_{q,l}(\vect{a})$ that intersects $I_p$. Without loss of generality, we may assume that $\Delta_{q,l}(\vect{a}) \subset I_p$. If this is not the case, we can replace $I_p$ with a slightly larger interval $I'_p$ such that $\Delta_{q,l}(\vect{a}) \subset I'_p$ and $m(I'_p) < 2 m(I_p)$ and proceed the same argument. Let us write $\Delta_{q,l}(\vect{a}) =  [s'-R^{-q+l}, s'+ R^{-q +l}]$ where $\vect{a} = (a_0, a_1, \dots, a_n) \in \Z^{n+1}\setminus\{\vect{0}\}$. For every $s_0 \in \Delta_{q,l}(\vect{a})$, let us denote
\[g_{\vect{r}}(q)U(\bphi(s_0))\vect{a} = \vect{v}(s_0) = (v_0(s_0), v_1(s_0), \dots, v_n(s_0) ). \]
Then we have that 
\begin{equation}
\label{equ:small-under-unipotent}
\max\{ \|\vect{v}(s_0)\|: s_0 \in \Delta_{q,l}(\vect{a}) \} < \rho 
\end{equation}
 and 
 \begin{equation}
 \label{equ:small-under-unipotent-start-become-large}
 \max\{ \|\vect{v}(s_0)\|: s_0 \in  [s'- R^{-q+l+1}, s'+R^{-q+l+1}] \} \geq \rho .
 \end{equation}
\par Recall that for $j =1, \dots, n$, $\lambda_j = \frac{1+ r_j}{1+ r_1}$. Let $1 \leq n' \leq n $ be the largest index $j$ such that $(1 -\lambda_j)q \leq l$.
\par For $s_0 \in  [s'-R^{-q+l} , s'+ R^{-q + l}]$, let us write $s_0 = s' + r R^{-q + l}$ for $r \in [-1, 1]$. As we explained before, we can approximate $\bphi(s_0)$ by $\bphi(s') +  R^{-q + l} r\bphi'(s')$. 
By our standing assumption on $\bphi$ (Standing Assumption \ref{assumption-2}), we have that $c_1 \leq |\varphi'_j(s_0)| \leq C_1$ for $j = 1, \dots, n$. By direct calculation, we have that 
\begin{align*}
 g_{\vect{r}}(q) U(\bphi(s_0))\vect{a} &=  g_{\vect{r}}(q)U(\bphi(s_0) - \bphi(s')) g_{\vect{r}}(-q) g_{\vect{r}}(q) U(\bphi(s'))\vect{a} \\ 
   &=  g_{\vect{r}}(q)U(r  R^{-q + l} \bphi'(s')) g_{\vect{r}}(-q) \vect{v}(s').
\end{align*} 
  Recall that $\vect{e}_i \in \R^n$ denote the vector with $i$th coordinate equal to $1$ and other coordinates equal to zero. By direct calculation, we have that 
\[
g_{\vect{r}}(q)U(r R^{-q +l} \bphi'(s')) g_{\vect{r}}(-q) = U\left(r  R^l \sum_{i=1}^n R^{-(1 - \lambda_i)q} \varphi'_i(s') \vect{e}_i \right).\]
Therefore, we have 
\begin{equation}
\label{equ:unipotent-action-on-vector}
\vect{v}(s_0) = U\left(r  R^l \sum_{i=1}^n R^{-(1 - \lambda_i)q} \varphi'_i(s') \vect{e}_i \right) \vect{v}(s').
\end{equation}
\par For the case $n' < n$, let us estimate 
\[U\left( - r  R^l \sum_{i=n'+1}^n R^{-(1 - \lambda_i)q} \varphi'_i(s') \vect{e}_i \right)\vect{v}(s_0). \] 
By our assumption, for $i \geq n' +1$, we have that $|r  R^l R^{-(1-\lambda_i)q}| \leq 1$. Therefore, if we write 
\begin{equation}
\label{equ:unipotent-action-on-vector-small-part}
U\left( - r  R^l \sum_{i=n'+1}^n R^{-(1 - \lambda_i)q} \varphi'_i(s') \vect{e}_i \right)\vect{v}(s_0) = \tilde{\vect{v}}(s_0) = (\tilde{v}_0(s_0), \tilde{v}_1(s_0), \dots, \tilde{v}_n(s_0)),
\end{equation}
where $\tilde{v}_0(s_0) = v_0(s_0) -r \sum_{i= n' +1}^n R^l R^{-(1-\lambda_i)q} \varphi'_i(s') v_i(s_0)$ and 
$\tilde{v}_i(s_0) = v_i(s_0)$ for $i=1,\dots, n$, then $|\tilde{v}_0(s_0)| < C = (n+1)C_1 \rho $, and $|\tilde{v}_i(s_0)| < \rho$ for $i = 1, \dots, n$. Let 
\[\vect{h} = \sum_{i=1}^{n'} R^{-(1 - \lambda_i)q } \varphi'_i(s') \vect{e}_i\]
 and 
\[\vect{h}_W = \sum_{i=1}^{n'} R^{-(1 - \lambda_i)q } \varphi'_i(s') \vect{w}_i \in W.\]
  Then $\|\vect{h}\|_2 = \|\vect{h}_W\|_2 \asymp 1$. Combining \eqref{equ:unipotent-action-on-vector} and \eqref{equ:unipotent-action-on-vector-small-part}, we have
\begin{equation} 
\label{equ:unipotent-action-on-vector-main-part}
U(r  R^l \vect{h} ) \vect{v}(s') = ( \tilde{v}_0(s_0), \tilde{v}_1(s_0), \dots, \tilde{v}_n(s_0)),
\end{equation}
where $|\tilde{v}_0(s_0)| < C $, and $|\tilde{v}_i(s_0)| < \rho$ for $i=1,\dots, n$. Let $E_{n'}$ be the subspace of $\R^n$ spanned by $\{\vect{e}_1, \dots, \vect{e}_{n'}\}$ and $W'_{n'}$ be the subspace of $W$ spanned by $\{\vect{w}_1, \dots, \vect{w}_{n'}\}$. Then $\vect{h} \in E_{n'}$. Let $\mathfrak{k} \in \SO(n)$ be an element such that $\mathfrak{k}\cdot \vect{e}_1 = \vect{h}$, $\mathfrak{k}\cdot E_{n'} = E_{n'}$, and $\mathfrak{k} \cdot \vect{e}_i = \vect{e}_i$ for $i = n' +1 , \dots, n$. Let $z(\mathfrak{k}) = \begin{bmatrix}1 & \\ & \mathfrak{k} \end{bmatrix} \in Z$. It is easy to see that $z(\mathfrak{k}) \vect{w}_+ = \vect{w}_+$, $z(\mathfrak{k}) \vect{w}_1 = \vect{h}_W$, $z(\mathfrak{k}) W'_{n'} = W'_{n'}$, and $z(\mathfrak{k}) \vect{w}_i = \vect{w}_i$ for $i = n' +1, \dots, n$. By the definition of $z(\mathfrak{k})$ and our discussion in \S \ref{subsec-canonical-representation-sln}, we have that $U(\vect{h}) = z(\mathfrak{k}) U(\|\vect{h}\|_2\vect{e}_1)z\inv(\mathfrak{k})$. Therefore, we have that $U(\vect{h}) \vect{h}_W =  \vect{h}_W + \|\vect{h}\|_2 \vect{w}_{+}$. Moreover, we have that $U(\vect{h}) \vect{w}_+ = \vect{w}_+$; for $i = 2, \dots, n'$, $U(\vect{h}) z(\mathfrak{k}) \vect{w}_i = z(\mathfrak{k})\vect{w}_i$; and for $i = n'+1, \dots, n$, $U(\vect{h}) \vect{w}_i = \vect{w}_i$. Let us write 
\[\vect{v}(s') = a_+(s') \vect{w}_+ + \sum_{i=1}^{n'} a_i(s') z(\mathfrak{k})\vect{w}_i + \sum_{i= n' +1}^n a_i(s') \vect{w}_i.\] 
Then the above discussion shows that 
\[U(r  R^l \vect{h})\vect{v}(s') = (a_+(s') + r  R^l a_1(s')) \vect{w}_+ + \sum_{i=1}^{n'} a_i(s') z(\mathfrak{k})\vect{w}_i + \sum_{i= n' +1}^n a_i(s') \vect{w}_i.\]
By \eqref{equ:small-under-unipotent}, \eqref{equ:unipotent-action-on-vector-small-part} and \eqref{equ:unipotent-action-on-vector-main-part}, we have that there exists a constant $C >0$ such that $|a_i(s')| < C$ for $i = 1, \dots, n$ and 
$|a_+(s') + r  R^l a_1(s')| < C $ for any $r \in [ -1, 1]$. This implies that $|a_+(s')| < C $, and $|a_1(s')| < C R^{-l} $. Therefore, we have that $\vect{v}(s') \in z(\mathfrak{k}) ([-C, C] \times [-C R^{-l}, C R^{-l}] \times [-C, C]^{n-1})$. 
\par Now let us estimate $|\mathcal{D}_{q,l}(I_p)|$. 

\par Suppose that $\mathcal{D}_{q,l}(I_p) = \{\Delta_{q,l}(\vect{a}_u) : 1 \leq u \leq L\}$. For each $u = 1, \dots, L$, let us take $s_u \in \Delta_{q,l}(\vect{a}_u)\cap I_p$ such that $s_u \in I_{q-1, u}$ for some $I_{q-1,u} \in \mathcal{I}_{q-1}$. Let us denote 
\[\vect{v}_u  = g_{\vect{r}}(q)U(\bphi(s_u))\vect{a}_u.\]
Then by our previous argument, we have that 
\begin{equation}
\label{equ:diophantine-condition-dangerous-interval}
\vect{v}_u  = a_{u, +} \vect{w}_+ + \sum_{i=1}^{n'} a_{u, i} z(\mathfrak{k})\vect{w}_i + \sum_{i=n'+1}^n a_{u,i} \vect{w}_i,
\end{equation}
where $|a_{u, +}| < C $, $|a_{u,1}| < C R^{-l}$, and $ |a_{u, i} | < C $ for $ i = 2, \dots, n$.
\par Now let us consider $g_{\vect{r}}(q)U(\bphi(s_1))\vect{a}_u$. Let us write $s_u = s_1 - r R^{-q + 2l}$ for some $r \in [-1,1]$. As we explained at the beginning of the proof, we can approximate $\bphi(I_p)$ by its linear part. Then we have that 
\begin{align*}
  g_{\vect{r}}(q)U(\bphi(s_1))\vect{a}_u  &= g_{\vect{r}}(q) U(\bphi(s_1) - \bphi(s_u)) g_{\vect{r}}(-q) g_{\vect{r}}(q) U(\bphi(s_u))\vect{a}_u \\
 &= g_{\vect{r}}(q) U(\bphi(s_1) - \bphi(s_u)) g_{\vect{r}}(-q) \vect{v}_u \\
 &= g_{\vect{r}}(q) U(r R^{-q + 2l} \bphi'(s)) g_{\vect{r}}(-q) \vect{v}_u \\ 
 &= U\left(r R^{2l} \sum_{i= 1}^n R^{-(1-\lambda_i)q} \varphi'_i (s) \vect{e}_i \right) \vect{v}_u.
\end{align*}
Let us denote $\vect{h} = \sum_{i=1}^{n'} R^{-(1-\lambda_i)q} \varphi'_i (s') \vect{e}_i$ as before. Then by \eqref{equ:diophantine-condition-dangerous-interval}, we have that 
\begin{align*}
g_{\vect{r}}(q)U(\bphi(s_1))\vect{a}_u & =   U(r R^{2l} \vect{h} + r R^{2l}\sum_{i = n' +1}^n R^{-(1-\lambda_i)q} \bphi'_i(s) \vect{e}_i)\vect{v}_u \\ 
 & =   \left(a_{u, +} + r  R^{2l} a_{u, 1} + r  R^{2l}\sum_{i= n' +1}^n R^{-(1-\lambda_i)q} \varphi'_i(s) a_{u, i} \right) \vect{w}_{+} \\
 &  + \sum_{i=1}^{n'} a_{u, i} z(\mathfrak{k})\vect{w}_i + \sum_{i=n'+1}^n a_{u,i} \vect{w}_i.
\end{align*}
Since $|a_{u,1}| \leq C R^{-l}$, and since for $i = n'+1 , \dots, n$, $(1-\lambda_i) q > l$, $|a_{u, i}| <C$, and $|\varphi'_i(x)|\leq C_1$, we have that 
\begin{align*}
 & \left| a_{u, +} + r  R^{2l} a_{u, 1} + r  R^{2l}\sum_{i= n' +1}^n R^{-(1-\lambda_i)q} \bphi'(s) a_{u, i}  \right|  \\
&\leq  |a_{u, +}| + |r| R^{2l } |a_{u,l}| +  |r| R^{2l } \sum_{i=n'+1}^n  R^{-(1-\lambda_i)q} | \bphi'(s)| |a_{u,i}| \\
&\leq  C + R^{2l } C R^{-l} + R^{2l } \sum_{i = n' + 1}^n R^{-l} C_1 C \\
&\leq   C + R^{2l } C R^{-l } + R^{2l } n R^{-l} C_1 C \\
 &\leq C_2 R^{l} \\
 \end{align*}
 where $ C_2 = 2C + n C_1 C >0$. This implies that for any $u = 1, \dots, L$, we have that 
 \[g_{\vect{r}}(q)U(\bphi(s_1))\vect{a}_u \in z(\mathfrak{k})([-C_2 R^l, C_2 R^l] \times [-C R^{-l}, C R^{-l}] \times [-C, C]^{n-1}).\]
 Let us consider the range of $g_{\vect{r}}(q -l) U(\bphi(s_1))\vect{a}_u= g_{\vect{r}}(-l) g_{\vect{r}}(q)U(\bphi(s_1))\vect{a}_u$. Let us write $g_{\vect{r}}(-l) = d_2 (l) d_1(l)$ where 
 \[d_1(l) = \begin{bmatrix} b^{-l} & & & & \\ & b^{r_1 l} \I_{ n'} & &  & \\ & & b^{r_{n' + 1} l} & & \\ & & & \ddots & \\ 
  & & & & b^{r_n l}\end{bmatrix},\]
  and 
  \[
  d_2(l) = \begin{bmatrix} 1 & & & & & \\ & 1 & & & & \\ & & b^{-( r_1 - r_2)l} & & & \\ & & & \ddots & & \\ 
  & & & & b^{-(r_1 - r_{n'})l} & \\ & & & & & \I_{n - n'}  \end{bmatrix}.
  \]
  Then we have that 
  \[g_{\vect{r}}(q -l) U(\bphi(s_1))\vect{a}_u \in d_2(l) d_1(l) z(\mathfrak{k})([-C_2 R^l, C_2 R^l] \times [-C R^{-l}, C R^{-l}] \times [-C, C]^{n-1}).\]
  By the definition of $z(\mathfrak{k})$, we have that $d_1(l) z(\mathfrak{k}) = z(\mathfrak{k}) d_1(l)$. Therefore, we have that 
  \begin{align*} 
  & d_1(l) z(\mathfrak{k})([-C_2 R^l, C_2 R^l] \times [-C R^{-l}, C R^{-l}] \times [-C, C]^{n-1}) \\
   &= z(\mathfrak{k}) d_1 (l) ([-C_2 R^l, C_2 R^l] \times [-C R^{-l}, C R^{-l}] \times [-C, C]^{n-1}) \\
   &= z(\mathfrak{k}) ([-C_2 b^{r_1 l} , C_2 b^{r_1 l}] \times [-C b^{-l} , C b^{-l} ]\times [-C b^{r_1 l} , C b^{r_1 l}]^{n_1 -1} \times \prod_{i=n'+1}^n [-C b^{r_i l} , C b^{r_i l}])  \\ 
   &\subset z(\mathfrak{k})([-C_2 b^{r_1 l} , C_2 b^{r_1 l}] \times [-1  , 1  ]\times [-C b^{r_1 l} , C b^{r_1 l}]^{n' -1} \times \prod_{i=n'+1}^n [-C b^{r_i l} , C b^{r_i l}]). 
   \end{align*}
   It is easy to see that 
   \[z(\mathfrak{k})([-C_2 b^{r_1 l} , C_2 b^{r_1 l}] \times [-1  , 1  ]\times [-C b^{r_1 l} , C b^{r_1 l}]^{n' -1} \times \prod_{i=n'+1}^n [-C b^{r_i l} , C b^{r_i l}])\]
   can be covered by a collection $\mathcal{B}$ of $ O(b^{\lambda l})$ balls of radius $1$ where $\lambda = n' r_1  + \sum_{i = n' +1}^n r_i$. Then we have that 
   \begin{align*} g_{\vect{r}}(q-l)U(\bphi(s_1))\vect{a}_u & \in d_2(l)\bigcup_{B \in \mathcal{B}} B \\
                                         & =  \bigcup_{B \in \mathcal{B}} d_2(l)B. 
                                         \end{align*}

    Since $d_2(l)$ is a contracting map, for every $B \in \mathcal{B}$, there exists a ball $B'$ of radius $C$ such that 
    $d_2(l) B  \subset B'$. Let $\mathcal{B}'$ denote the collection of all such $B'$'s. Then we have that 
    \[g_{\vect{r}}(q-l)U(\bphi(s_1))\vect{a}_u \in \bigcup_{B' \in \mathcal{B}'} B'.\]

    Since $g_{\vect{r}}(q-l)U(\bphi(s_1))\vect{a}_u \in g_{\vect{r}}(q-l)U(\bphi(s_1))\Z^{n+1}$, we have that 
    \[g_{\vect{r}}(q-l)U(\bphi(s_1))\vect{a}_u \in \bigcup_{B' \in \mathcal{B}'} B' \cap \Lambda,\]
    where $\Lambda = g_{\vect{r}}(q-l)U(\bphi(s_1))\Z^{n+1}$. By our assumption, $s_1 \in I_{q-1, 1}$ for some $I_{q-1, 1} \in \mathcal{I}_{q-1}$. This implies that $s_1 \in I_{q-l}$ for some $I_{q-l} \in \mathcal{I}_{q-l}$. Therefore, $\Lambda = g_{\vect{r}}(q-l) U(\bphi(s_1))\Z^{n+1} \in K_{\kappa}$, i.e., $\Lambda$ does not contain any nonzero vectors with norm $\leq \kappa$. Therefore, there exists a constant $C_4$ such that every ball of radius $1$ contains at most $C_4 \kappa^{-n-1} = C_4 R^{(n+1)k}$ points in $\Lambda$. Thus, we have that 
    \begin{align*} \sharp\mathcal{D}_{q,l}(I_p) = \sharp\{g_{\vect{r}}(q-l)U(\bphi(s_1))\vect{a}_u : 1 \leq u \leq L\} & \leq  \sum_{B' \in \mathcal{B}'} \sharp(B'\cap \Lambda) \\ 
     & \leq  \sum_{B' \in \mathcal{B}'} C_4 R^{(n+1)k} \\ 
        & \leq  C_5 b^{\lambda l + 4n k} \leq C_5 b^{(\lambda + \frac{1}{200n})l}, 
       \end{align*} 
     where $C_5 = C_3 C_4$ and $\lambda = n' r_1  + \sum_{i= n' +1}^n r_i$. Now let us estimate $\lambda$. In fact, 
     \begin{align*}
     \lambda  & = \sum_{i=1}^n r_i  + \sum_{i=1}^{n'} (r_1 - r_i)\\ 
                                 & = 1  + \sum_{i=1}^{n'} (r_1 - r_i) .
     \end{align*}
     By our assumption, for $i=1, \dots, n'$, we have that $r_1 - r_i \leq \frac{l}{q} \leq \frac{1}{100 n^2}$. Therefore, we have that 
     \[\lambda \leq 1 + n \frac{1}{100n^2} = 1+ \frac{1}{100 n}.\]
     Thus, we have that 
     \[\sharp\mathcal{D}_{q,l}(I_p) \leq C_5 b^{(1+ \frac{1}{100n} + \frac{1}{200n})l} \leq C_5 R^{(1- \frac{1}{10n})l}.\]
     The last inequality above holds because $b = R^{\frac{1}{1+r_1}} \leq R^{\frac{n}{n+1}}$.
     \par This completes the proof.

\end{proof} 
%%%%%%%%%%%%%%%%%%%%%%%%%%%%%%%%%%%%%%%%%%%%%%%%%%%%%%%%%%%%%%%%%%%%%%%%%%%%%%%%%%%%%%%%%%%%%%%%%%%%%%%%%%%%%%%%%%%%
%%%%%%%%%%%%%%%%%%%%%%%%%%%%%%%%%%%%%%%%%%%%%%%%%%%%%%%%%%%%%%%%%%%%%%%%%%%%%%%%%%%%%%%%%%%%%%%%%%%%%%%%%%%%%%%%%%%%
%%%%%%%%%%%%%%%%%%%%%%%%%%%%%%%%%%%%%%%  Proof of the main result %%%%%%%%%%%%%%%%%%%%%%%%%%%%%%%%%%%%%%%%%%%%%%%%%%
%%%%%%%%%%%%%%%%%%%%%%%%%%%%%%%%%%%%%%%%%%%%%%%%%%%%%%%%%%%%%%%%%%%%%%%%%%%%%%%%%%%%%%%%%%%%%%%%%%%%%%%%%%%%%%%%%%%%
%%%%%%%%%%%%%%%%%%%%%%%%%%%%%%%%%%%%%%%%%%%%%%%%%%%%%%%%%%%%%%%%%%%%%%%%%%%%%%%%%%%%%%%%%%%%%%%%%%%%%%%%%%%%%%%%%%%%

\section{Proof of the main result}
\label{sec-proof-main-result}
\par In this section we will finish the proof of Proposition \ref{prop:distance-sequence-small}. By our discussion in \S \ref{sec-introduction} and \S \ref{sec:cantor-like-construction}, Proposition \ref{prop:distance-sequence-small} implies Theorem \ref{thm:main-task}, and thus Theorem \ref{thm:main-thm} and Theorem \ref{thm:double-intersection}. 
\par The structure of the section is as follows. In the first subsection, we will prove Proposition \ref{prop:distance-sequence-small} for the case $q \leq 10^6 n^4 N k$. The second, third and fourth subsections are devoted to the proof for the case $q > 10^6 n^4 N k$. The key point is to estimate $F(\hat{\mathcal{I}}_{q, p}, I_p)$ for $I_p \in \mathcal{I}_p$. The second subsection deals with the case $p = q  - 4000 n^2 N k$. The third subsection deals with the case $p = q - 2l$ where $ 2000 n^2 N k < l < 2 \eta' q$. The fourth subsection deals with the case $p = 0$. 
\par The third and fourth subsections contain some technical results on the canonical representation of $\SL(n+1, \R)$ on $\bigwedge^i V$ for $i = 2, \dots, n$. They are also main technical contributions of this paper.

\par Our basic tool is the following non-divergence theorem due to Kleinbock:
\begin{theorem}[see~{\cite[Theorem 2.2]{Kleinbock2008}}]
\label{thm:non-divergence-dani-margulis}
There exist constants $C, \alpha >0$ such that the following holds: For any $g \in \mathrm{SL}(n+1,\mathbb{R})$, any one parameter unipotent subgroup $U = \{u(r) : r \in \R\} \subset \SL(n+1, \R)$ and any $R >0$, if for any $i= 1,2,\dots, n$ and any $\vect{v} = \vect{v}_1 \wedge \cdots \wedge \vect{v}_i \in \bigwedge^i \Z^{n+1} \setminus \{\vect{0}\}$, 
\[\max\{ \|u(r)g\vect{v}\|: r \in [-R, R] \} \geq \rho^i,\]
then for any $0 < \epsilon <\rho$,
\[ m \left( \{ r \in [-R,R] : u(r) g \Z^{n+1} \notin K_{\epsilon} \} \right) \leq C \left( \frac{\epsilon}{\rho} \right)^{\alpha} R. \]
\end{theorem}

\par We will also need the following important result due to Kleinbock and Margulis \cite{Klein_Mar}.
\begin{theorem}[see~{\cite[Proposition 2.3]{Klein_Mar}}]
\label{thm:kleinbock-margulis-nondegenerate-manifold}
Let $\bphi: [0,1] \to \R^n$ be a $C^n$ non-degenerate curve. Then there exists a constant $\alpha >0$ such that for any $s \in [0,1]$ there exists an interval $J$ centered at $s$ and positive constants $D$ and $\rho$ such that for any $t \geq 0$ and $0 < \epsilon < \rho$ one has
\[ m\left(\{ s' \in J: g_{\vect{r}}(t) u(\bphi(s'))\Z^{n+1} \not\in K_{\epsilon}  \} \right) \leq D \left(\frac{\epsilon}{\rho}\right)^{\alpha} m(J). \]
\end{theorem}
\begin{remark}
\label{rmk:kleinbock-margulis-nondegenerate-manifold}
\par The exact statement in \cite[Proposition 2.3]{Klein_Mar} is more general than the above theorem. For example, the statement holds for any $C^n$ differentiable non-degenerate submanifolds.
\end{remark}
\par From Theorem \ref{thm:kleinbock-margulis-nondegenerate-manifold}, one can easily deduce the following corollary:
\begin{corollary}
\label{cor:nondegenerate-submanifold-quantitative-estimate}
Let $\bphi: [0,1] \to \R^n$ be a $C^n$ non-degenerate curve. Then there exist constants $C>0$, $\alpha >0$ and $0<\rho_1<1$ such that for any $t \geq 0$ and $0 < \epsilon < \rho_1$ one has
\[ m\left(\{ s \in [0,1]: g_{\vect{r}}(t) u(\bphi(s))\Z^{n+1} \not\in K_{\epsilon}  \}\right) \leq C \left(\frac{\epsilon}{\rho_1}\right)^{\alpha}. \]
\end{corollary}
\begin{proof}
\par For any $s \in [0,1]$, one can find the corresponding interval $J = J(s)$, constants $D(s) >0$ and $\rho(s) >0$ arising from Theorem \ref{thm:kleinbock-margulis-nondegenerate-manifold}. Then $\{J(s): s \in [0,1]\}$ is an open covering of $[0,1]$. Since $[0,1]$ is compact, there is a finite covering $\{J(s_i): i = 1, 2, \dots, M\}$. Without loss of generality, we may assume that $m(J(s_i)) \leq 2$. Let us choose $\rho_1 := \min \{ \rho(s_i): i = 1,2, \dots, M\}$ and $C := 2 M \max\{D(s_i): i=1,2,\dots, M\} $. Then for any $t \geq 0$ and $0 <\epsilon < \rho_1$, we have that 
\[ E_{t, \epsilon} \subset \bigcup_{i=1}^M E_{t, \epsilon}\cap J(s_i) \]
where $E_{t, \epsilon} := \{ s \in [0,1]: [g_{\vect{r}}(t) u(\bphi(s))] \not\in K_{\epsilon}  \}$. By Theorem \ref{thm:kleinbock-margulis-nondegenerate-manifold}, for any $i=1,2,\dots, M$, we have that 
\[ m(E_{t, \epsilon}\cap J(s_i)) \leq D(s_i)\left(\frac{\epsilon}{\rho(s_i)}\right)^{\alpha} m(J(s_i)) \leq D(s_i) \left(\frac{\epsilon}{\rho_1}\right)^{\alpha} \cdot 2 .\]
Therefore, we have that
\[
	m(E_{t, \epsilon}) \leq \sum_{i=1}^M 2 D(s_i) \left(\frac{\epsilon}{\rho_1}\right)^{\alpha} \leq C \left(\frac{\epsilon}{\rho_1}\right)^{\alpha}.
\]
This completes the proof.
\end{proof}

\par Later in this paper, we will choose $0< \rho < 1$ such that $C\left(\frac{2 \rho}{\rho_1}\right)^{\alpha} < \frac{1}{1000}$.

%%%%%%%%%%%%%%%%%%%%%%%%%%%%%%%%%%%%%%%%%%%%%%%%%%%%%%%%%%%%%%%%%%%%%%%%%%%%%%%%%%%%%%%%%%%%%%%%%%%%%%%%%%%%%%%%%%%
%%%%%%%%%%%%%%%%%%%%%%%%%%%%%%%%%%%%%%%%%%%%%%%%%%%%%%%%%%%%%%%%%%%%%%%%%%%%%%%%%%%%%%%%%%%%%%%%%%%%%%%%%%%%%%%%%%%
%%%%%%%%%%%%%%%%%%%%%%%%%%%%%%%%%%%%%%%%    Case q small    %%%%%%%%%%%%%%%%%%%%%%%%%%%%%%%%%%%%%%%%%%%%%%%%%%%%%%%
%%%%%%%%%%%%%%%%%%%%%%%%%%%%%%%%%%%%%%%%%%%%%%%%%%%%%%%%%%%%%%%%%%%%%%%%%%%%%%%%%%%%%%%%%%%%%%%%%%%%%%%%%%%%%%%%%%%
%%%%%%%%%%%%%%%%%%%%%%%%%%%%%%%%%%%%%%%%%%%%%%%%%%%%%%%%%%%%%%%%%%%%%%%%%%%%%%%%%%%%%%%%%%%%%%%%%%%%%%%%%%%%%%%%%%%

\subsection{The case where $q$ is small}
\label{subsec-case-q-small}
\par In this subsection, let us assume that $q \leq 10^6 n^4 N k$. Then $\hat{\mathcal{I}}_{q,0} = \hat{I}_q$ and $\hat{\mathcal{I}}_{q,p} = \emptyset$ for other $p$.
\begin{proposition}
\label{prop:small-portion-q-small}

\[ F(\hat{\mathcal{I}}_{q, 0}, I) \ll R^{q - \alpha k}.\]
\end{proposition}
\begin{proof}
By Corollary \ref{cor:nondegenerate-submanifold-quantitative-estimate}, we have that for any $\kappa = R^{-k} >0$ such that $2\kappa < \rho$, the following holds:
\[ m(\{s \in [0,1]: g_{\vect{r}}(q)U(\bphi(s))\Z^{n+1} \notin K_{2\kappa} \}) \leq C \left(\frac{2\kappa}{\rho}\right)^{\alpha}. \]
\par On the other hand, by the definition of $\hat{\mathcal{I}}_q$, for any $I_q \in \hat{\mathcal{I}}_q$, there exists $s \in I_q$ such that 
\[ g_{\vect{r}}(q)U(\bphi(s)) \Z^{n+1} \in X \setminus K_{ \kappa}. \] 
Since $g_{\vect{r}}(q)U(\bphi(I_q)) \Z^{n+1}$ is contained in $1$-neighborhood of $g_{\vect{r}}(q)U(\bphi(s)) \Z^{n+1}$, we have
 \[g_{\vect{r}}(q)U(\bphi(I_q)) \Z^{n+1} \subset X \setminus K_{2 \kappa}. \] 
 Therefore, we have that 
\begin{align*}
F(\hat{\mathcal{I}}_{q,0}, I) R^{-q} & =  m\left(\bigcup_{I_q \in \hat{\mathcal{I}}_q} I_q \right) \\ 
                                     & =  m(\{s \in I: g_{\vect{r}}(q)U(\bphi(s))\Z^{n+1} \notin K_{2\kappa} \}) \leq C_6 \kappa^{\alpha} = C_6 R^{-\alpha k}
\end{align*} 
where $C_6 = C \left( \frac{2}{\rho} \right)^{\alpha}$. This finishes the proof.
\end{proof}
Let us choose $R >1$ such that $R^{\alpha} > 1000^{10^6 n^4 N}$. 
\begin{proof}[Proof of Proposition \ref{prop:distance-sequence-small} for $q \leq 10^6 n^4 N k$]
\par It suffices to show that 
\[ \left( \frac{4}{R} \right)^q F(\hat{\mathcal{I}}_{q,0}, I) \]
can be arbitrarily small. In fact, by Proposition \ref{prop:small-portion-q-small}, we have that
\begin{align*}
	\left( \frac{4}{R} \right)^q F(\hat{\mathcal{I}}_{q,0}, I) & = \left( \frac{4}{R} \right)^q O(R^{q - \alpha k})  \\
	 & = O\left(\frac{4^q}{ R^{\alpha k}}\right)   =  O\left(\frac{4^{10^6 n^4 N k}}{R^{\alpha k}}\right)  = O\left(\left(\frac{4}{1000}\right)^{10^6 n^4 N k}\right) .
\end{align*}
Then it is easy to see that $\left( \frac{4}{R} \right)^q F(\hat{\mathcal{I}}_{q,0}, I) \to 0$ as $k \to \infty$.
\par This completes the proof for $q \leq 10^6 n^4 N k$.
\end{proof}
%%%%%%%%%%%%%%%%%%%%%%%%%%%%%%%%%%%%%%%%%%%%%%%%%%%%%%%%%%%%%%%%%%%%%%%%%%%%%%%%%%%%%%%%%%%%%%%%%%%%%%%%%%%%%%%
%%%%%%%%%%%%%%%%%%%%%%%%%%%%%%%%%%%%%%%%%%%%%%%%%%%%%%%%%%%%%%%%%%%%%%%%%%%%%%%%%%%%%%%%%%%%%%%%%%%%%%%%%%%%%%%
%%%%%%%%%%%%%%%%%%%%%%%%%%%%%%%%      Generic case     %%%%%%%%%%%%%%%%%%%%%%%%%%%%%%%%%%%%%%%%%%%%%%%%%%%%%%%%
%%%%%%%%%%%%%%%%%%%%%%%%%%%%%%%%%%%%%%%%%%%%%%%%%%%%%%%%%%%%%%%%%%%%%%%%%%%%%%%%%%%%%%%%%%%%%%%%%%%%%%%%%%%%%%%
%%%%%%%%%%%%%%%%%%%%%%%%%%%%%%%%%%%%%%%%%%%%%%%%%%%%%%%%%%%%%%%%%%%%%%%%%%%%%%%%%%%%%%%%%%%%%%%%%%%%%%%%%%%%%%%

\subsection{The generic case}
\label{subsec-generic-case}
\par The rest of the section is devoted to the proof of Proposition \ref{prop:distance-sequence-small} for $q > 10^6 n^4 N k$. In the following subsections, we will estimate $F(\hat{\mathcal{I}}_{q,p}, I_p)$ for different $p$'s. In this subsection we will estimate $F(\hat{\mathcal{I}}_{q,p}, I_p)$ for $p = q - 4000 n^2 N k$. We call it {\bf the generic case}.
\begin{proposition}
\label{prop:small-portion-q-large-p-large}
Let $q > 10^6 n^4 N k$ and $p = q- 4000 n^2 N k$. Then for any $I_p \in \mathcal{I}_p$, we have that 
\[F(\hat{\mathcal{I}}_{q,p}, I_p) \ll R^{q-p - \alpha k} .\]
\end{proposition}
\begin{proof}
\par Let us fix $I_p \in \mathcal{I}_p$. If $F(\hat{\mathcal{I}}_{q,p}, I_p) = 0$, then the statement trivially holds. 
\par Suppose $F(\hat{\mathcal{I}}_{q,p}, I_p) >0$, let us take $I_q \in \hat{\mathcal{I}}_{q,p}$ and $s \in I_q\cap I_p$. Without loss of generality, we may assume that $[s- R^{-q+ 2000 n^2 N k}, s + R^{-q + 2000 n^2 N k}] \subset I_p$. If this is not the case, we can replace $I_p$ with a slightly larger interval $I'_p \supset I_p$ such that $[s- R^{-q+ 2000 n^2 N k}, s + R^{-q + 2000 n^2 N k}] \subset I'_p$ and $m(I'_p) < 2 m(I_p)$ and proceed the same argument. Then for any $i = 1, \dots, n$ and $\vect{v} = \vect{v}_1 \wedge \cdots \wedge \vect{v}_i \in \bigwedge^i \Z^{n+1}\setminus\{\vect{0}\}$, we have that 
\[\max\{\|g_{\vect{r}}(q) U(\bphi(s')) \vect{v}\|: s' \in [s- R^{-q+ 2000 n^2 N k}, s + R^{-q + 2000 n^2 N k}]\} \geq \rho^i.\]
 Therefore, we have that 
\[\max\{ \|g_{\vect{r}}(q)U(\bphi(s')) \vect{v}\| : s' \in I_p \} \geq \rho^i.\]
\par On the other hand, as we explained in the proof of Proposition \ref{prop:count-dangerous-intervals}, we can approximate $\bphi(I_p)$ by its linear part, that is to say, for any $s' \in I_p$, we approximate $\bphi(s')$ by $ \bphi(s) + (s'-s)\bphi'(s)$. For $s' \in I_p$, let us write $s' = s + r R^{-q + 4000 n^2 N k}$ where $r \in [-1,1]$ and denote $g =  g_{\vect{r}}(q) U(\bphi(s))$. Then
\begin{align*}
 g_{\vect{r}}(q)U(\bphi(s')) & =  g_{\vect{r}}(q) U(\bphi(s') - \bphi(s)) g_{\vect{r}}(-q) g_{\vect{r}}(q) U(\bphi(s)) \\
                               & =  g_{\vect{r}}(q) U(r R^{-q + 4000 n^2 N k} \bphi'(s)) g_{\vect{r}}(-q) g \\
                               & =  U(r R^{4000 n^2 N k} \vect{h}) g,
\end{align*}
where $\vect{h} = \varphi'_1(s) \vect{e}_1 + \sum_{i=2}^n R^{-(1-\lambda_i)q} \varphi'_i(s) \vect{e}_i$. Recall that $\lambda_i = \frac{1+r_i}{1+ r_1}$. Since $\{U(rR^{4000 n^2 N k} \vect{h} ): r \in \R\}$ is a one parameter unipotent subgroup, by Theorem \ref{thm:non-divergence-dani-margulis}, we have that
\[m(\{ r \in [-1,1]: U(r R^{4000 n^2 N k} \vect{h}) g \Z^{n+1} \notin K_{2\kappa}\}) \leq 2C \left(\frac{2\kappa}{\rho}\right)^{\alpha} . \]
This implies that 
\[m(\{ s \in I_p: g_{\vect{r}}(q)U(\bphi(s))\Z^{n+1} \notin K_{2\kappa}\}) \leq 2C \left(\frac{2\kappa}{\rho}\right)^{\alpha} m(I_p). \]
\par On the other hand, it is easy to see that $g_{\vect{r}}(q) U(\bphi(I_q))\Z^{n+1} \subset X \setminus K_{2\kappa}$ for any $I_q \in \hat{\mathcal{I}}_q$. Therefore we have that
\[\begin{array}{cl}
 & F(\hat{\mathcal{I}}_{q, p}, I_p) R^{-q} \\
 \leq & m(\{ s \in I_p: g_{\vect{r}}(q)U(\bphi(s))\Z^{n+1} \notin K_{2\kappa}\}) \\
 \leq & 2C \left( \frac{2\kappa}{\rho} \right)^{\alpha} m(I_p) \\
 =  & 2 C \left(\frac{2}{\rho}\right)^{\alpha} \kappa^{\alpha} R^{-p}  = C_7 R^{-p - \alpha k}
\end{array}\]
where $C_7 = 2 C \left(\frac{2}{\rho}\right)^{\alpha}$. This proves the statement.
\end{proof}
\par By Proposition \ref{prop:small-portion-q-large-p-large}, we have that for $p = q - 4000 n^2 N k$ and any $I_p \in \mathcal{I}_p$, the following holds: 
\begin{equation}
\label{equ:q-large-p-large}
 \left(\frac{4}{R}\right)^{q-p} F(\hat{\mathcal{I}}_{q,p}, I_p)  \ll \left(\frac{4}{R}\right)^{q-p} R^{q-p - \alpha k} 
   =   \frac{4^{4000 n^2 N k}}{R^{\alpha k}} = \left( \frac{4}{1000}\right)^{4000 n^2 N k}.
\end{equation}
Then it is easy to see that $\left(\frac{4}{R}\right)^{q-p} F(\hat{\mathcal{I}}_{q,p}, I_p) \to 0$ as $ k \to \infty$.

%%%%%%%%%%%%%%%%%%%%%%%%%%%%%%%%%%%%%%%%%%%%%%%%%%%%%%%%%%%%%%%%%%%%%%%%%%%%%%%%%%%%%%%%%%%%%%%%%%%%%%%%%%%%%%%%%
%%%%%%%%%%%%%%%%%%%%%%%%%%%%%%%%%%%%%%%%%%%%%%%%%%%%%%%%%%%%%%%%%%%%%%%%%%%%%%%%%%%%%%%%%%%%%%%%%%%%%%%%%%%%%%%%%
%%%%%%%%%%%%%%%%%%%%%%%%%%%%%%%%%%%%%%%%%  Dangerous Case  %%%%%%%%%%%%%%%%%%%%%%%%%%%%%%%%%%%%%%%%%%%%%%%%%%%%%%
%%%%%%%%%%%%%%%%%%%%%%%%%%%%%%%%%%%%%%%%%%%%%%%%%%%%%%%%%%%%%%%%%%%%%%%%%%%%%%%%%%%%%%%%%%%%%%%%%%%%%%%%%%%%%%%%%
%%%%%%%%%%%%%%%%%%%%%%%%%%%%%%%%%%%%%%%%%%%%%%%%%%%%%%%%%%%%%%%%%%%%%%%%%%%%%%%%%%%%%%%%%%%%%%%%%%%%%%%%%%%%%%%%%

\subsection{Dangerous case}
\label{subsec-dangerous-case}
\par In this subsection, we will consider the case where $2000 n^2 N k < l < 2\eta' q$ and $p = q - 2l$. We call this case {\bf the $(q,l)$-dangerous case}.
\begin{proposition}
\label{prop:small-portion-dangerous-case}
For any $I_p \in \mathcal{I}_p$, we have that 
\[ F(\hat{\mathcal{I}}_{q,p}, I_p) \ll R^{ q - p -\frac{l}{20n} }. \]
\end{proposition}
\par Let us recall that for $1000 n^2 N k < l' < \eta' q$, a $(q,l')$-dangerous interval $\Delta_{q,l'}(\vect{a})$ associated with a nonzero integer vector $\vect{a} \in \Z^{n+1}$ is a closed interval of the form 
\[\Delta_{q,l'}(\vect{a}) = [ s - R^{-q + l'}, s + R^{-q + l'}]\]
such that $I_q \subset \Delta_{q,l'}(\vect{a})$ for some $I_q \in \hat{\mathcal{I}}_q$,
\[ \max\{ \|g_{\vect{r}}(q)U(\bphi(s'))\vect{a}\|: s' \in \Delta_{q,l'}(\vect{a}) \} < \rho \]
 and
 \[ \max\{ \|g_{\vect{r}}(q)U(\bphi(s'))\vect{a}\|: s' \in [ s - R^{-q + l'+1}, s + R^{-q + l'+1}] \} \geq \rho. \]
\par The following lemma is crucial to prove Proposition \ref{prop:small-portion-dangerous-case} and is one of the main technical contributions of this paper:
\begin{lemma}
\label{lm:key-lemma}
For any $i = 1, \dots, n$ and $I_q \in \hat{\mathcal{I}}_{q,p}(i)$ intersecting $I_p$, one of the following two cases holds:
\begin{enumerate}[label=\textit{Case \arabic*.}]
\item there exists a $(q,l')$-dangerous interval $\Delta_{q, l'} (\vect{a})$ containing $I_q$ for some $ l/2 \leq l' \leq l$;
\item there exists $s \in I_q$ and 
\[\vect{v} = \vect{v}_1 \wedge \cdots \wedge \vect{v}_i \in \bigwedge\nolimits^i \Z^{n+1}\setminus \{\vect{0}\}\]
such that if we write 
\[ g_{\vect{r}}(q)U(\bphi(s))\vect{v} = \vect{w}_+ \wedge \vect{w}^{(i-1)} + \vect{w}^{(i)} \]
where $\vect{w}^{(i-1)} \in \bigwedge^{i-1} W$ and $\vect{w}^{(i)} \in \bigwedge^i W$, then we have that $\|\vect{w}_+ \wedge \vect{w}^{(i-1)}\| = \| \vect{w}^{(i-1)} \| < \rho^i$ and $\|\vect{w}^{(i)}\| \leq \rho^i R^{-l/2}$.
\end{enumerate}

\end{lemma}
\begin{proof}
\par If $i = 1$, then the first case holds. We may assume that $i \geq 2$.
\par By the definition of $\hat{\mathcal{I}}_{q,p}(i)$, there exists $\vect{v} = \vect{v}_1 \wedge \cdots \wedge \vect{v}_i \in \bigwedge^i \Z^{n+1}\setminus \{\vect{0}\}$ such that for any $s \in I_q$, 
\[ \max\{ \|g_{\vect{r}}(q)U(\bphi(s')) \vect{v}\|: s' \in [s- R^{-q+l}, s+ R^{-q +l} ] \} < \rho^i \]
and 
\[ \max\{ \|g_{\vect{r}}(q)U(\bphi(s')) \vect{v}\|: s' \in [s- R^{-q+l+1}, s+ R^{-q +l+1} ] \} \geq \rho^i. \]
\par Without loss of generality, we may assume that the sublattice $L_i$ generated by $\{\vect{v}_1, \dots , \vect{v}_i \}$ is a primitive $i$-dimensional sublattice of $\Z^{n+1}$. Then $\Lambda_i = g_{\vect{r}}(q) U(\bphi(s)) L_i$ is a primitive $i$-dimensional sublattice of $\Lambda = g_{\vect{r}}(q) U(\bphi(s))\Z^{n+1}$. For simplicity, let us denote $g = g_{\vect{r}}(q)U(\bphi(s))$. Let us choose the Minkowski reduced basis $\{g\vect{v}'_1, \dots, g\vect{v}'_i\}$ of $\Lambda_i$. Since 
\[ d(\Lambda_i) = \|g\vect{v}\| < \rho^i,\]
we have that $\|g\vect{v}'_1\| < \rho$ by the Minkowski Theorem. 
\par Let us repeat the argument in the proof of Proposition \ref{prop:count-dangerous-intervals}. Recall that 
for $j = 1,\dots, n $, $\lambda_j = \frac{1+ r_j}{1+ r_1}$. Let $1 \leq n' \leq n$ be the largest index $j$ such that $(1- \lambda_j) q \leq l$. By Standing Assumption \ref{assumption-2}, we have that $c_1 \leq |\bphi'_i (s)| \leq C_1$ for any $i = 1, \dots, n$ and $s \in [0,1]$. Fix any $s \in I_q$ and let $\vect{h} = \sum_{i=1}^{n'} R^{-(1-\lambda_i)q} \bphi'(s)\vect{e}_i$. For any $s' \in [ s - R^{-q + l}, s + R^{-q + l}]$, let us write $s' = s + r R^{-q + l}$ where $ r \in [-1, 1]$. By the same argument as in the proof of Proposition \ref{prop:count-dangerous-intervals}, we have that 
\[ g_{\vect{r}}(q)U(\bphi(s')) = U(O(1)) U(r R^l \vect{h}) g_{\vect{r}}(q)U(\bphi(s)) = U(O(1)) U(r R^l \vect{h}) g. \]
Therefore, we have that 
\[ \| U(r R^l \vect{h}) g\vect{v} \| < \rho^i \]
for any $r \in [-1,1]$.
\par Following the notation in the proof of Proposition \ref{prop:count-dangerous-intervals}, let us denote 
$\vect{h} = \mathfrak{k} \cdot \vect{e}_1$ for $\mathfrak{k} \in \SO(n)$ and 
\[ z(\mathfrak{k}) = \begin{bmatrix}1 & \\ & \mathfrak{k} \end{bmatrix} \in Z.\]
For $ j = 1, \dots, i$, let us write 
\[g \vect{v}'_j = a_+ (j) \vect{w}_+ + a_1(j) z(\mathfrak{k})\vect{w}_1 + \vect{w}'(j) \]
where $\vect{w}' (j) \in z(\mathfrak{k}) W_2$. Then 
\begin{align*}
	g \vect{v}  & =  (g\vect{v}'_1)\wedge \cdots \wedge (g\vect{v}'_i) \\
	            & =  \bigwedge_{j=1}^i (a_+ (j) \vect{w}_+ + a_1(j) z(\mathfrak{k})\vect{w}_1 + \vect{w}'(j) ) \\
	            & =  \vect{w}_+ \wedge (z(\mathfrak{k})\vect{w}_1) \wedge \left( \sum_{j <j'} \epsilon_{+,1}(j,j') a_+(j) a_1(j') \bigwedge_{k \neq j,j'} \vect{w}'(k) \right) \\ 
	            &  + \vect{w}_+ \wedge \left( \sum_{j=1}^i \epsilon_+(j) a_+(j) \bigwedge_{k \neq j} \vect{w}'(k) \right) + (z(\mathfrak{k})\vect{w}_1)\wedge \left( \sum_{j=1}^i \epsilon_1(j) a_1(j) \bigwedge_{k \neq j} \vect{w}'(k) \right) \\ 
	            &  + \bigwedge_{j=1}^i \vect{w}'(j) 
	\end{align*}
where $\epsilon_{+,1}(j,j'), \epsilon_+(j), \epsilon_1(j) \in \{\pm 1\}$ for every $j,j' \in \{1, \dots, i\}$. By our discussion in \S \ref{subsec-canonical-representation-sln} on the representation of $\SL(2, \vect{h})$ on $\bigwedge^i V$, we have that 
\begin{align*}
U( r R^l \vect{h})g \vect{v} & =  \vect{w}_+ \wedge (z(\mathfrak{k})\vect{w}_1) \wedge \left( \sum_{j <j'} \epsilon_{+,1}(j,j') a_+(j) a_1(j') \bigwedge_{k \neq j,j'} \vect{w}'(k) \right) \\ 
	            &  + \vect{w}_+ \wedge \left( \sum_{j=1}^i \epsilon_+(j) a_+(j) \bigwedge_{k \neq j} \vect{w}'(k) \right) \\
	            &  +  r R^l \vect{w}_+ \wedge \left( \sum_{j=1}^i \epsilon_1(j) a_1(j) \bigwedge_{k \neq j} \vect{w}'(k) \right) \\ 
	            &   + (z(\mathfrak{k})\vect{w}_1)\wedge \left( \sum_{j=1}^i \epsilon_1(j) a_1(j) \bigwedge_{k \neq j} \vect{w}'(k) \right) + \bigwedge_{j=1}^i \vect{w}'(j) .
\end{align*}
Since $\|U(r R^l \vect{h}) g \vect{v}\| < \rho^i$ for any $r \in [-1,1]$, we have that 
\[ \left\|\sum_{j=1}^i \epsilon_1(j) a_1(j) \bigwedge_{k \neq j} \vect{w}'(k)\right\| \leq \rho^i R^{-l}. \]
Let us consider the following two cases:
\begin{enumerate}
	\item $|a_1(1)| \leq R^{-l/2}$.
	\item $|a_1(1)| > R^{-l/2}$.
\end{enumerate}
\par Let us first suppose $|a_1(1)| \leq R^{-l/2}$. Note that $\|g\vect{v}'_1\| < \rho$. Then by repeating the calculation in the proof of Proposition \ref{prop:count-dangerous-intervals}, we conclude that 
\[ \max\{ \|g_{\vect{r}}(q) U(\bphi(s')) \vect{v}'_1\| : s' \in [ s- R^{-q + l/2}, s + R^{-q + l/2} ] \}  < \rho. \]
On the other hand, by our definition on $\hat{\mathcal{I}}_{q, p}(i)$, we have that
\[\max\{ \|g_{\vect{r}}(q)U(\bphi(s'))\vect{v}'_1\|: s' \in [ s - R^{-q+l+1}, s + R^{-q+l+1}] \} \geq \rho. \] 
This implies that $I_q \subset \Delta_{q,l'}(\vect{v}'_1)$ for some $l/2 \leq l' \leq l$. This proves the first part of the statement.
\par Now let us suppose $|a_1(1)| > R^{-l/2}$. Then we have that 
\begin{align*}
\epsilon_1(1) a_1(1) \bigwedge_{j=1}^i \vect{w}'(j) & =  \vect{w}'(1)\wedge \left( \epsilon_1(1) a_1(1) 
\bigwedge_{k\neq 1} \vect{w}'(k) \right) \\ 
                          & =  \vect{w}'(1)\wedge \left( \sum_{j=1}^i \epsilon_1(j) a_1(j) \bigwedge_{k \neq j} \vect{w}'(k) \right).
\end{align*}
Therefore, we have that 
	\begin{align*}
	|a_1 (1)| \left\|\bigwedge_{j=1}^i \vect{w}'(j)\right\| & =  \left\| \vect{w}'(1)\wedge \left( \sum_{j=1}^i \epsilon_1(j) a_1(j) \bigwedge_{k \neq j} \vect{w}'(k) \right) \right\| \\
	& \leq  \|\vect{w}'(1)\| \left\| \sum_{j=1}^i \epsilon_1(j) a_1(j) \bigwedge_{k \neq j} \vect{w}'(k) \right\| \\
	& \leq \rho \cdot \rho^i R^{-l} = \rho^{i+1} R^{-l}.
	\end{align*}
Since $|a_1(1)| > R^{-l/2}$ and $\rho < 1$, we have that 
\[ \left\|\bigwedge_{j=1}^i \vect{w}'(j)\right\| \leq \rho^{i} R^{-l/2}. \]
If we write 
\[ g \vect{v} = \vect{w} \wedge \vect{w}^{(i-1)} + \vect{w}^{(i)} \]
where $\vect{w}^{(i-1)} \in \bigwedge^{i-1} W$ and $\vect{w}^{(i)} \in \bigwedge^i W$, then 
\[ \vect{w}^{(i)} =  (z(\mathfrak{k})\vect{w}_1)\wedge \left( \sum_{j=1}^i \epsilon_1(j) a_1(j) \bigwedge_{k \neq j} \vect{w}'(k) \right) + \bigwedge_{j=1}^i \vect{w}'(j). \]
By our previous argument, we have that 
\[ \|\vect{w}^{(i)}\| \leq \rho^i R^{-l/2}. \]
This proves the second part of the statement.
\end{proof}
\par The following lemma takes care of the second case of Lemma \ref{lm:key-lemma}.
\begin{lemma}
\label{lm:2nd-case-key-lemma}
Let $i \in \{2,\dots, n \}$. Let $\mathcal{D}_{q,p}(I_p, i)$ denote the collection of $I_q \in \hat{\mathcal{I}}_{q,p}$ intersecting $I_p$ and not contained in any $(q,l')$-dangerous interval for any $l/2 \leq l' \leq l$. Let 
\[D_{q,p}(I_p, i) := \bigcup_{I_q \in \mathcal{D}_{q,p}(I_p, i)} I_q.\] 
Then for any closed subinterval $J \subset I_p$ of length $R^{-q + (1+ \frac{1}{2n})l}$, we have that 
\[ m(D_{q,p}(I_p, i) \cap J) \ll R^{-\frac{l}{20n}} m(J). \]
\end{lemma}
\begin{proof}
\par Let us fix a closed subinterval $J \subset I_p$ of length $R^{-q + (1+ \frac{1}{2n})l}$.
\par For any $s \in I_q \in \mathcal{D}_{q,p}(I_p, i)$, there exists $\vect{v} = \vect{v}_1 \wedge \cdots \wedge \vect{v}_i \in \bigwedge^i \Z^{n+1}\setminus \{\vect{0}\}$ such that 
\[ \max \{ \|g_{\vect{r}}(q) U(\bphi(s'))\vect{v}\| : s' \in [ s - R^{-q+l}, s + R^{-q + l}] \} < \rho^i.\]
\par Let us denote the interval $[s - R^{-q+l}, s + R^{-q+l}]$ by $\Delta_{q,l}(\vect{v}, i)$. Then every $I_q \in \mathcal{D}_{q,l}(I_p,i)$ is contained in some $\Delta_{q,l}(\vect{v},i)$ and every $\Delta_{q,l}(\vect{v}, i)$ contains at most $O(R^l)$ different $I_q \in \mathcal{D}_{q,l}(I_p, i)$.
\par We will follow the notation used in the proof of Lemma \ref{lm:key-lemma}. Let $g = g_{\vect{r}}(q)U(\bphi(s))$, $\vect{h} = \mathfrak{k} \cdot \vect{e}_1$ and 
\[z(\mathfrak{k}) = \begin{bmatrix} 1 & \\ & \mathfrak{k} \end{bmatrix} \in Z\]
be as in the proof of Lemma \ref{lm:key-lemma}. For $j = 1 , \dots, i$, let us write 
\begin{align*}
g\vect{v}_j & = a_+(j)\vect{w}_+ + a_1(j) z(\mathfrak{k})\vect{w}_1 + \vect{w}'(j) \\
            & = a_+(j)\vect{w}_+ + \vect{w}(j) 
\end{align*}
where $\vect{w}'(j) \in z(\mathfrak{k})W_2$ and $\vect{w}(j) = a_1(j) z(\mathfrak{k})\vect{w}_1 + \vect{w}'(j) \in W$. Then 
\begin{align*}
g \vect{v} & =  \vect{w}_+ \wedge (z(\mathfrak{k})\vect{w}_1) \wedge \left( \sum_{j <j'} \epsilon_{+,1}(j,j') a_+(j) a_1(j') \bigwedge_{k \neq j,j'} \vect{w}'(k) \right) \\ 
	            &  + \vect{w}_+ \wedge \left( \sum_{j=1}^i \epsilon_+(j) a_+(j) \bigwedge_{k \neq j} \vect{w}'(k) \right)  \\ 
	            &  + (z(\mathfrak{k})\vect{w}_1)\wedge \left( \sum_{j=1}^i \epsilon_1(j) a_1(j) \bigwedge_{k \neq j} \vect{w}'(k) \right) + \bigwedge_{j=1}^i \vect{w}'(j) .
\end{align*}
 By Lemma \ref{lm:key-lemma}, we have that 
 \[ \left\| (z(\mathfrak{k})\vect{w}_1)\wedge \left( \sum_{j=1}^i \epsilon_1(j) a_1(j) \bigwedge_{k \neq j} \vect{w}'(k) \right) \right\| \leq \rho^i R^{-l}\]
 and 
 \[
 \left\| \bigwedge_{j=1}^i \vect{w}'(j) \right\| \leq \rho^i R^{-l/2}.
 \]
 Let us take the collection of all possible $\Delta_{q,l}(\vect{v}, i)$'s intersecting $J$, say 
 \[ \{  \Delta_{q,l}(\vect{v}(M), i) = [ s(M) - R^{-q +l}, s(M) + R^{-q + l}] : M = 1, \dots, L \}. \]
For simplicity, let us denote $g(M) = g_{\vect{r}}(q) U(\bphi(s(M)))$ for $M = 1,\dots, L$. Since $\bphi(J)$ can be approximated by its linear part, we have that the corresponding $\vect{h}$ and $\mathfrak{k}$ for $s(M)$ is the same for $M = 1, \dots, L$. Then
\[ 
g(M) \vect{v}(M)  =  \vect{w}_+ \wedge \vect{w}^{(i-1)}(M)  + (z(\mathfrak{k})\vect{w}_1) \wedge (\vect{w}')^{ (i-1)}(M) + \vect{w}^{(i)}(M)	          
\]
where $\vect{w}^{(i-1)}(M) \in \bigwedge^{i-1}W$, $(\vect{w}')^{ (i-1)}(M) \in \bigwedge^{i-1} z(\mathfrak{k})W_2$ and $\vect{w}^{(i)}(M) \in \bigwedge^{i} z(\mathfrak{k})W_2$. By our previous discussion, we have that 
\[ \left\| \vect{w}_+ \wedge \vect{w}^{(i-1)}(M) \right\| < \rho^i, \]
\[ \|(\vect{w}')^{ (i-1)}(M) \| =  \left\| (z(\mathfrak{k})\vect{w}_1) \wedge (\vect{w}')^{ (i-1)}(M) \right\| \leq \rho^i R^{-l},\]
and 
\[ \left\| \vect{w}^{(i)}(M) \right\| \leq \rho^i R^{-l/2}.\]
Now let us consider $g(1) \vect{v}(M)$. Let us write $s(1) - s(M) = r R^{-q + (1 + \frac{1}{2n})l}$ where $r \in [ -1, 1]$. By our previous discussion, we have that 
\begin{align*}
g(1) = g_{\vect{r}}(q) U(\bphi(s(1))) & =  U(O(1)) U(r R^{(1+ \frac{1}{2n})l} \vect{h}) g_{\vect{r}}(q) U(\bphi(s(M))) \\ 
           & =  U(O(1)) U(r R^{(1+ \frac{1}{2n})l} \vect{h}) g(M).
\end{align*}
Therefore, we have that 
\[
	g(1)\vect{v}(M)   =  U(O(1)) U(r R^{(1+ \frac{1}{2n})l} \vect{h}) g(M) \vect{v}(M).
\]
\par It is easy to see that we can ignore the contribution of $U(O(1))$ and identify $g(1)\vect{v}(M)$ with 
$U(r R^{(1+ \frac{1}{2n})l} \vect{h}) g(M) \vect{v}(M)$. Then we have that 
	\begin{align*}
	g(1)\vect{v}(M) & =  U(r R^{(1+ \frac{1}{2n})l} \vect{h}) g(M) \vect{v}(M) \\
	 & =  \vect{w}_+ \wedge \vect{w}^{(i-1)}(M) + r R^{(1+\frac{1}{2n})l} \vect{w}_+ \wedge (\vect{w}')^{ (i-1)}(M) \\ 
	 &  + (z(\mathfrak{k})\vect{w}_1) \wedge (\vect{w}')^{ (i-1)}(M) + \vect{w}^{(i)}(M) .
	\end{align*}
\par Now let us look at the range of 
\[g_{\vect{r}}(-l/2) g(1) \vect{v}(M) = g_{\vect{r}}(q-l/2) U(\bphi(s(1))) \vect{v}(M).\]
It is easy to see that $g_{\vect{r}}(-l/2)\vect{w}_+ = b^{-l/2} \vect{w}_+$, $\|g_{\vect{r}}(-l/2)z(\mathfrak{k})\vect{w}_1\| \leq b^{r_1 l/2} \|z(\mathfrak{k})\vect{w}_1\|$,
\[\|g_{\vect{r}}(-l/2)\vect{w}^{(i-1)}(M)\| \leq b^{l/2} \|\vect{w}^{(i-1)}(M)\|,\]
\[ \|g_{\vect{r}}(-l/2)(\vect{w}')^{ (i-1)}(M) \| \leq b^{(1-r_1)l/2} \|(\vect{w}')^{ (i-1)}(M) \|, \]
and 
\[ \|g_{\vect{r}}(-l/2) \vect{w}^{(i)}(M)\| \leq b^{l/2} \| \vect{w}^{(i)}(M)\|. \]
Since 
	\begin{align*}
	g_{\vect{r}}(-l/2) g(1) \vect{v}(M) & =  b^{-l/2} \vect{w}_+ \wedge (g_{\vect{r}}(-l/2)\vect{w}^{(i-1)}(M)) \\ &  +
	r R^{(1+\frac{1}{2n})l} b^{-l/2} \vect{w}_+ \wedge(g_{\vect{r}}(-l/2) (\vect{w}')^{ (i-1)}(M) ) \\
	&  + (g_{\vect{r}}(-l/2)z(\mathfrak{k})\vect{w}_1) \wedge (g_{\vect{r}}(-l/2)(\vect{w}')^{ (i-1)}(M)) \\
	& + g_{\vect{r}}(-l/2) \vect{w}^{(i)}(M),
	\end{align*}
we have that 
	\begin{align*}
	\|g_{\vect{r}}(-l/2) g(1) \vect{v}(M)\| & \leq  b^{-l/2} \|\vect{w}_+ \wedge (g_{\vect{r}}(-l/2)\vect{w}^{(i-1)}(M))\| \\
	&  + R^{(1+\frac{1}{2n})l} b^{-l/2} \|\vect{w}_+ \wedge(g_{\vect{r}}(-l/2) (\vect{w}')^{ (i-1)}(M) )\| \\
	&  + \|g_{\vect{r}}(-l/2)z(\mathfrak{k})\vect{w}_1\| \cdot \|g_{\vect{r}}(-l/2)(\vect{w}')^{ (i-1)}(M)\| \\
	&  + \|g_{\vect{r}}(-l/2) \vect{w}^{(i)}(M)\|  \\
	& \leq b^{-l/2} b^{l/2} \|\vect{w}^{(i-1)}(M))\| +  R^{(1+\frac{1}{2n})l} b^{-l/2}b^{(1-r_1)l/2} \|(\vect{w}')^{ (i-1)}(M) \| \\
	&  + b^{r_1 l/2} \|z(\mathfrak{k})\vect{w}_1\| \cdot b^{(1-r_1)l/2} \|(\vect{w}')^{ (i-1)}(M) \| + b^{l/2}\| \vect{w}^{(i)}(M)\| \\
	& \leq  b^{-l/2} b^{l/2} \rho^i + R^{(1+\frac{1}{2n})l} b^{-l/2}b^{(1-r_1)l/2} \rho^i R^{-l} \\
	& + b^{r_1 l/2} b^{(1-r_1)l/2} \rho^i R^{-l} + b^{l/2} \rho^i R^{-l/2} \\
	& \leq  \rho^i + \rho^i + \rho^i R^{-l/2} + \rho^i \leq 1.
	\end{align*}
For $M = 1, \dots, L$, let $\Lambda_i(\vect{v}(M))$ denote the $i$-dimensional primitive sublattice of $\Z^{n+1}$ corresponding to $\vect{v}(M)$. We will apply Proposition \ref{prop:counting-sublattices} to estimate $L$. Thus, let us keep the notation used there. By the inequality above, we have that $g_{\vect{r}}(-l/2)g(1)\Lambda_i(\vect{v}(M)) \in \mathcal{C}_i(g_{\vect{r}}(-l/2)g(1)\Z^{n+1}, 1)$ for every $M = 1, \dots, L$. On the other hand, since 
$x(1) \in I_q \in \hat{\mathcal{I}}_q$, we have that 
\[ g_{\vect{r}}(-l/2)g(1)\Z^{n+1} = g_{\vect{r}}(q - l/2) U(\bphi(s(1)))\Z^{n+1} \in K_{\kappa}. \]
By Proposition \ref{prop:counting-sublattices}, we have that 
\[ L \leq  \sharp\mathcal{C}_i(g_{\vect{r}}(-l/2)g(1)\Z^{n+1}, 1) \leq \kappa^{-N} = R^{Nk}.\]
Therefore, we have that 
\begin{align*} 
m(D_{q,p}(I_p, i) \cap J) & \leq  L R^{-q+ l} \leq R^{-q + l + N k} \\ 
                          & \leq   R^{-q + l + \frac{l}{100 n}} \leq R^{- \frac{l}{20n}} R^{-q + (1 + \frac{1}{2n})l} =  R^{- \frac{l}{20n}} m(J).
\end{align*} 
\par This completes the proof.
\end{proof}
\par Lemma \ref{lm:2nd-case-key-lemma} easily implies the following:
\begin{corollary}
\label{cor:2nd-case-key-lemma}
Let us keep the notation as above. Then 
\[m(D_{q, p} (I_p, i)) \ll R^{-\frac{l}{20n}} m(I_p).\]
\end{corollary}
\begin{proof}
\par The statement follows from Lemma \ref{lm:2nd-case-key-lemma} by dividing $I_p$ into subintervals of length $R^{-q + (1 + \frac{1}{2n})l}$.
\end{proof}
\par Now we are ready to prove Proposition \ref{prop:small-portion-dangerous-case}.
\begin{proof}[Proof of Proposition \ref{prop:small-portion-dangerous-case}]
\par Let us fix $I_p \in \mathcal{I}_p$. For every $l/2 \leq l' \leq l$, let us denote by $D_{q,l'}(I_p)$ denote the union of $(q, l')$-dangerous intervals intersecting $I_p$. By Proposition \ref{prop:count-dangerous-intervals}, we have that $m(D_{q,l'}(I_p)) = O\left(R^{-\frac{l'}{10n}}\right)m(I_p)$. Therefore, we have that 
\begin{align*}
	m\left( \bigcup_{l/2 \leq l' \leq l} D_{q,l'}(I_p) \right) & \leq  \sum_{l/2 \leq l' \leq l} m( D_{q,l'}(I_p)) \\ 
	& \ll  \sum_{l/2 \leq l' \leq l} R^{-\frac{l'}{10n}} m(I_p) \\ 
	& \ll  R^{-\frac{l}{20n}} m(I_p).
\end{align*}
By Corollary \ref{cor:2nd-case-key-lemma}, we have that 
\begin{align*} 
m \left(\bigcup_{i = 2}^{n} D_{q, p} (I_p, i)\right) & \leq \sum_{i = 2}^{n} m(D_{q, p} (I_p, i)) \\
                                                       & \ll \sum_{i=2}^{n} R^{-\frac{l}{20n}} m(I_p) \ll R^{-\frac{l}{20n}} m(I_p)
 \end{align*}
\par By Lemma \ref{lm:key-lemma}, we have that 
\[I_q \subset \bigcup_{l/2 \leq l' \leq l} D_{q,l'}(I_p)\cup\bigcup_{i = 2}^{n} D_{q, p} (I_p, i) \] 
for any $I_q \in \hat{\mathcal{I}}_{q,p}$.
Therefore, we have that 
\begin{align*}
 F(\hat{\mathcal{I}}_{q,p}, I_p) R^{-q} & \leq  m\left( \bigcup_{l/2 \leq l' \leq l} D_{q,l'}(I_p)\bigcup_{i = 2}^{n} D_{q, p} (I_p, i) \right) \\
               & \leq  m\left( \bigcup_{l/2 \leq l' \leq l} D_{q,l'}(I_p) \right) + m \left(\bigcup_{i = 2}^{n} D_{q, p} (I_p, i)\right) \\
               & \ll  R^{-\frac{l}{20n}} m(I_p) = R^{-p - \frac{l}{20n}}.
 \end{align*}
 This proves that 
 \[ F(\hat{\mathcal{I}}_{q,p}, I_p) \ll R^{q-p - \frac{l}{20n}}. \]
\end{proof}
\par By Proposition \ref{prop:small-portion-dangerous-case}, we have that 
\begin{align}
\label{equ:dangerous-case}
 \sum_{l = 2000 n^2 N k}^{2\eta' q} \left(\frac{4}{R}\right)^{2l} \max_{I_{q- 2l} \in \mathcal{I}_{q- 2l}} F(\hat{\mathcal{I}}_{q, q - 2l}, I_{q - 2l}) &\ll  \sum_{l = 2000 n^2 N k}^{2\eta' q} \left(\frac{4}{R}\right)^{2l} R^{2l - \frac{l}{20n}} \\
  & \leq  \sum_{l = 2000 n^2 N k}^{2\eta' q} \left( \frac{16}{1000}\right)^l \ll \left( \frac{16}{1000}\right)^{2000 n^2 N k}.
\end{align}
From this it is easy to see that 
\begin{equation}
\label{equ:dangerous-case-final}
\sum_{l = 2000 n^2 N k}^{2\eta' q} \left(\frac{4}{R}\right)^{2l} \max_{I_{q- 2l} \in \mathcal{I}_{q- 2l}} F(\hat{\mathcal{I}}_{q, q - 2l}, I_{q - 2l}) \to 0
\end{equation}
as $k \to \infty$.
%%%%%%%%%%%%%%%%%%%%%%%%%%%%%%%%%%%%%%%%%%%%%%%%%%%%%%%%%%%%%%%%%%%%%%%%%%%%%%%%%%%%%%%%%%%%%%%%%%%%%%%%%%%%%%%%
%%%%%%%%%%%%%%%%%%%%%%%%%%%%%%%%%%%%%%%%%%%%%%%%%%%%%%%%%%%%%%%%%%%%%%%%%%%%%%%%%%%%%%%%%%%%%%%%%%%%%%%%%%%%%%%%
%%%%%%%%%%%%%%%%%%%%%%%%%%%%%%%%%%%%  Extremely dangerous case   %%%%%%%%%%%%%%%%%%%%%%%%%%%%%%%%%%%%%%%%%%%%%%%
%%%%%%%%%%%%%%%%%%%%%%%%%%%%%%%%%%%%%%%%%%%%%%%%%%%%%%%%%%%%%%%%%%%%%%%%%%%%%%%%%%%%%%%%%%%%%%%%%%%%%%%%%%%%%%%%
%%%%%%%%%%%%%%%%%%%%%%%%%%%%%%%%%%%%%%%%%%%%%%%%%%%%%%%%%%%%%%%%%%%%%%%%%%%%%%%%%%%%%%%%%%%%%%%%%%%%%%%%%%%%%%%%

\subsection{Extremely dangerous case}
\label{subsec-extremely-dangerous-case}
\par In this subsection we will estimate $F(\hat{\mathcal{I}}_{q,0}, I)$. We call this case {\bf the extremely dangerous case}.
\begin{proposition}
\label{prop:small-portion-extremely-dangerous}
There exists a constant $\nu >0$ such that for any $q > 10^6 n^4 N k$, we have that 
\[ F(\hat{\mathcal{I}}_{q,0}, I) \ll R^{(1-\nu) q}. \]
\end{proposition}
\par Similarly to Lemma \ref{lm:key-lemma}, we have the following:
\begin{lemma}
\label{lm:key-lemma-extremely-dangerous}
\par For any $i = 1, \dots, n$ and $I_q \in \hat{\mathcal{I}}_{q,0}(i)$, one of the following two cases holds:
\begin{enumerate}[label=\textit{Case \arabic*.}]
 \item there exists a $q$-extremely dangerous interval $\Delta_q(\vect{a})$ such that $I_q \in \Delta_q(\vect{a})$;
 \item there exists $\vect{v} = \vect{v}_1 \wedge \cdots \wedge \vect{v}_i \in \bigwedge^i \Z^{n+1} \setminus \{\vect{0}\}$ such that the following holds:
for any $s \in I_q$, if we write 
\[g_{\vect{r}}(q)U(\bphi(s))\vect{v} = \vect{w}_+ \wedge \vect{w}^{(i-1)} + \vect{w}^{(i)} \]
where $\vect{w}^{(i-1)} \in \bigwedge^{i-1} W$ and $\vect{w}^{(i)} \in \bigwedge^i W$, then $\|\vect{w}_+ \wedge \vect{w}^{(i-1)}\| \leq \rho^i$ and $\|\vect{w}^{(i)}\| \leq \rho^i R^{-\eta' q}$.
\end{enumerate}
\end{lemma}
\begin{proof}
\par The proof is the same as the proof of Lemma \ref{lm:key-lemma}. In fact, the argument in the proof of Lemma \ref{lm:key-lemma} works for $l= 2\eta' q$ and thus concludes the statement.
\end{proof}

 \begin{definition}
 \label{def:extremely-dangerous-second-case}
 \par For $ i = 2, \dots, n$, let $\mathcal{D}_q(i) $ denote the collection of $I_q \in \hat{\mathcal{I}}_{q,0}(i)$ such that the second case in Lemma \ref{lm:key-lemma-extremely-dangerous} holds and let 
\[ D_q(i) := \bigcup_{I_q \in \mathcal{D}_q(i)} I_q .\]
Moreover, for $I_q \in \mathcal{D}_q(i)$, let $\vect{v} =\vect{v}_1 \wedge \cdots \wedge \vect{v}_i \in \bigwedge^i \Z^{n+1} \setminus \{\vect{0}\}$ be the vector given in the second case of Lemma \ref{lm:key-lemma-extremely-dangerous}. Then for $s \in I_q$, we can write 
\[g_{\vect{r}}(q)U(\bphi(s))\vect{v} = \vect{w}_+ \wedge \vect{w}^{(i-1)} + \vect{w}^{(i)} \]
as in the second case of Lemma \ref{lm:key-lemma-extremely-dangerous}. For $l \geq \eta' q$, let $\mathcal{D}'_{q,l}(i)$ denote the collection of $I_q \in \mathcal{D}_q(i)$ such that 
\[ \rho^i R^{-l+1} \leq \|\vect{w}^{(i)}\| \leq \rho^i R^{-l},\]
and let 
\[D'_{q,l}(i) : = \bigcup_{I_q \in \mathcal{D}'_{q,l}(i)} I_q .\]
\end{definition}
\begin{lemma}
\label{lm:key-lemma-extremely-dangerous-2nd-case}
There exists a constant $\nu >0$ such that for any $q > 10^6 n^4 N k$ and any $i= 2, \dots, n$, we have that 
\[ m(D_q(i)) \ll R^{- \nu q}. \]
\end{lemma}
\begin{proof}
\par For any $ \eta' q \leq  l \leq 2\eta' q $, using the same argument as in the proof of Lemma \ref{lm:2nd-case-key-lemma}, we can prove that 
\[ m(D'_{q,l}(i)) \ll R^{-\frac{l}{20n}}. \]
Therefore, we have that 
\begin{align*}
	m\left( \bigcup_{l = \eta' q}^{2\eta' q} D'_{q,l}(i) \right) & \leq \sum_{l = \eta' q}^{2\eta' q} m(D'_{q,l}(i)) \\ 
	 & \ll  \sum_{l = \eta' q}^{2\eta' q} R^{-\frac{l}{20n}} \ll R^{-\frac{\eta' q}{ 20 n}}.
\end{align*}
\par Let us denote 
\[\mathcal{D}'_q(i) := \bigcup_{l > 2\eta' q} \mathcal{D}'_{q,l}\]
 and 
 \[D'_q(i) : = \bigcup_{I_q \in \mathcal{D}'_q(i)} I_q.\]
 Then it is enough to show that 
 \[ m(D'_q(i)) \ll R^{-\nu q}.\]
\par For any $I_q \in \mathcal{D}'_q(i)$ and $s \in I_q$, there exists $\vect{v} = \vect{v}_1 \wedge \cdots \wedge \vect{v}_i \in \bigwedge^i \Z^{n+1} \setminus \{\vect{0}\}$ such that if we write 
\[ g_{\vect{r}}(q) U(\bphi(s))\vect{v} = \vect{w}_+ \wedge \vect{w}^{(i-1)} + \vect{w}^{(i)} \]
where $\vect{w}^{(i-1)} \in \bigwedge^{i-1} W$ and $\vect{w}^{(i)} \in \bigwedge^i W$, then we have that 
$\|\vect{w}_+ \wedge \vect{w}^{(i-1)}\| \leq \rho^i$ and $\|\vect{w}^{(i)}\| \leq \rho^i R^{-2\eta' q}$.
\par Recall that $\eta = (1+r_1)\eta'$. Let us deal with the following two cases separately:
\begin{enumerate}
\item $r_n \geq \frac{\eta}{n}$.
\item There exists $ 1 < n_1 \leq n $ such that for $r_i \geq \frac{\eta}{n} $ for $1 \leq i < n_1$ and 
$r_i < \frac{\eta}{n}$ for $ n_1 \leq i \leq n$.
\end{enumerate}
\par Let us first deal with the first case. For this case, let us define
\[ g^{\eta}(t) := \begin{bmatrix} b^{-\eta t} & \\ & b^{\eta t/n}\I_n \end{bmatrix} \in \SL(n+1, \R) \]
and $g_{\vect{r}, \eta}(t) := g^{\eta}(t) g_{\vect{r}}(t)$. It is easy to see that 
\[g^{\eta}(t) \vect{w}_+ = b^{-\eta t} \vect{w}_+ = R^{-\eta' t} \vect{w}_+,\]
and 
\[g^{\eta}(t) \vect{w} = b^{\eta t/n} \vect{w} = R^{\eta' t/n} \vect{w}\]
 for any $\vect{w} \in W$. 
\par Then we have that 
\begin{align*}
	\| g_{\vect{r}, \eta}(q)U(\bphi(s))\vect{v}\| & = \|g^{\eta}(q)(\vect{w}_+ \wedge \vect{w}^{(i-1)} + \vect{w}^{(i)})\| \\ 
	& \leq \|g^{\eta}(q)(\vect{w}_+ \wedge \vect{w}^{(i-1)})\| + \|g^{\eta}(q) \vect{w}^{(i)}\| \\
	& =  b^{-\eta q(1- \frac{i-1}{n})} \|\vect{w}_+ \wedge \vect{w}^{(i-1)}\| + b^{\frac{\eta q i}{n}} \|\vect{w}^{(i)}\| \\
	& \leq  b^{-\frac{\eta q}{n}} \rho^i + b^{\eta q} R^{-2\eta' q} \rho^i \leq R^{-\frac{\eta' q}{n}} \rho^i.
\end{align*}
By the Minkowski Theorem, the above inequality implies that the lattice $g_{\vect{r}, \eta}(q)U(\bphi(s))\Z^{n+1}$ contains a nonzero vector with norm $\leq R^{-\frac{\eta' q}{n^2}} \rho$. Therefore, for any $I_q \in \mathcal{D}'_q(i)$ we have that
\[  g_{\vect{r}, \eta}(q)U(\bphi(I_q))\Z^{n+1} \not\in K_{\sigma } \]
where $\sigma = R^{-\frac{\eta' q}{n^2}} \rho$. Then by Corollary \ref{cor:nondegenerate-submanifold-quantitative-estimate}, we have that 
\[ m\left( \{ s \in I: g_{\vect{r},\eta}(q) U(\bphi(s))\Z^{n+1} \not\in K_{\sigma} \} \right) \ll \sigma^{\alpha} =  R^{-\frac{\alpha\eta' q}{n^2}}.\]
This proves that 
\[ m(D'_q (i)) \ll  R^{-\frac{\alpha\eta' q}{n^2}} . \]
\par This finishes the proof for the first case.
\par Now let us take care of the second case. Let us denote 
\[ \xi(t) := \begin{bmatrix} b^{-\beta t} & & & & & & \\ & 1 & & & & & \\  & & \ddots & & & & \\ & & & 1 & & & \\ & & & & b^{r_{n_1} t} & & \\ & & & & & \ddots & \\ & & & & & & b^{r_n t} \end{bmatrix} \in \SL(n+1, \R) \]
where $\beta = \sum_{j= n_1}^n r_j < \eta$ and 
\[g' (t) := \xi(t) g_{\vect{r}}(t) = \begin{bmatrix} b^{\chi t} & & & & & & \\ & b^{-r_1 t} & & & & & \\ & & \ddots & & & & \\ & & & b^{- r_{n_1 - 1} t} & & & \\ & & & & 1 & &  \\ & & & & & \ddots & \\ & & & & & & 1 \end{bmatrix}\]
where $\chi = \sum_{j=1}^{n_1 - 1} r_j$. Then it is easy to see that 
\begin{align*}
\xi(t) \vect{w}_+ &= b^{-\beta t} \vect{w}_+ , \\
\xi(t) \vect{w}_j &= \vect{w}_j
\end{align*}
for $j = 1, \dots, n_1 - 1$, and 
\[\xi(t) \vect{w}_j = b^{r_j t} \vect{w}_j\]
for $j = n_1, \dots, n$. Then we have that 
	\begin{align*}
	\|g'(q) U(\bphi(s)) \vect{v}\| & =  \|\xi(q) (\vect{w}_+ \wedge \vect{w}^{(i-1)} + \vect{w}^{(i)}) \| \\
	 & \leq \|\xi(q)(\vect{w}_+ \wedge \vect{w}^{(i-1)})\| + \|\xi(q) \vect{w}^{(i)}\| \\
	 & \leq \| \vect{w}_+ \wedge \vect{w}^{(i-1)} \| + b^{\beta q} \|\vect{w}^{(i)}\| \\
	 & \leq  \rho^i + b^{\beta q} R^{-2\eta' q} \rho^i \\
	 & \leq   \rho^i + b^{\eta q} R^{-2\eta' q} \rho^i \leq \rho^i + R^{-\eta' q} \rho^i < (2\rho)^i.
	\end{align*}
Moreover, for any $s' \in \Delta(s):= [ s - R^{-q(1-2\eta')} , s + R^{-q(1-2\eta')}]$, we also have that 
\[ \|g'(q) U(\bphi(s')) \vect{v}\| < (2\rho)^i. \]
Let $C > 0$ and $\alpha >0$ be the constants given in Theorem \ref{thm:non-divergence-dani-margulis}. Then by the Minkowski Theorem, the inequality above implies that for any $s' \in \Delta(s)$, the lattice $g'(q) U(\bphi(s'))\Z^{n+1}$ contains a nonzero vector of length $< 2\rho$. Let $\vect{v}_{s'} \in \Z^{n+1}\setminus \{\vect{0}\}$ be the vector such that 
$\|g'(q)U(\bphi(s'))\vect{v}_{s'}\| < 2 \rho$. Let us write 
\[\vect{v}_{s'} = (v_{s'}(0), v_{s'}(1),\dots, v_{s'}(n)).\]
Then for $j = n_1 , \dots, n$, we have that $|v_{s'}(j)| < 2\rho$. Therefore, $v_{s'}(j) = 0$ for any $j = n_1 , \dots, n$. In other words, $\vect{v}_{s'}$ is contained in the subspace spanned $\{\vect{w}_+ , \vect{w}_1, \dots, \vect{w}_{n_1 - 1}\}$. For notational simplicity, let us denote this subspace by $\R^{n_1}$ and denote the set of integer points contained in the subspace by $\Z^{n_1}$. Accordingly, let us denote by $\SL(n_1, \R)$ the subgroup
\[ \left\{ \begin{bmatrix} X & \\ & \I_{n+1 - n_1} \end{bmatrix}: X \in \SL(n_1, \R) \right\} \subset \SL(n+1, \R) \]
and denote by $\SL(n_1, \Z)$ the subgroup of integer points in $\SL(n_1,\R)$. Note that $g'(q) \in \SL(n_1, \R)$. $U(\bphi(s'))$ can also be considered as an element in $\SL(n_1,\R)$ since it preserves $\R^{n_1}$. Then $\|g'(q)U(\bphi(s'))\vect{v}_{s'}\| < 2\rho$ implies that for any $s' \in \Delta(s)$, the lattice $g'(q)U(\bphi(s'))\Z^{n_1}$ contains a nonzero vector of length $< 2\rho$. Let $K_{2\rho}(n_1) \subset X(n_1) = \SL(n_1, \R)/\SL(n_1, \Z)$ denote the set of unimodular lattices in $\R^{n_1}$ which do not contain any nonzero vector of length $< 2 \rho$. Then the claim above implies that 
\[ m(\{s' \in \Delta(s): g'(q) U(\bphi(s')) \Z^{n_1} \not\in K_{2\rho}(n_1) \} ) = m(\Delta(s)). \]
By Theorem \ref{thm:non-divergence-dani-margulis}, there exist $j \in {1, \dots, n_1 -1}$ and $\vect{v}' = \vect{v}'_1 \wedge \cdots \wedge \vect{v}'_j \in \bigwedge^j \Z^{n_1} \setminus\{\vect{0}\}$ such that 
\begin{equation}
\label{equ:reduce-dimension} 
\max\{ \|g'(q)U(\bphi(s'))\vect{v}'\|: s' \in [s- R^{-q(1-2\eta')}, s+ R^{-q(1-2\eta')}] \} < \rho_1^j
\end{equation}
since otherwise we will have that
\[m(\{s' \in \Delta(x): g'(q) U(\bphi(s')) \Z^{n_1} \not\in K_{2\rho}(n_1) \} ) \leq C \left(\frac{2 \rho}{\rho_1}\right)^{\alpha} m(\Delta(s)) < \frac{1}{1000} m(\Delta(s)).\]
Now we have \eqref{equ:reduce-dimension} in dimension $n_1$ and every weight of $g'(q)$ is at least $\eta/n$. Then we can repeat the argument for the first case with $n+1$ replaced by $n_1$ to complete the proof.
\end{proof}
\par Now we are ready to prove Proposition \ref{prop:small-portion-extremely-dangerous}.
\begin{proof}[Proof of Proposition \ref{prop:small-portion-extremely-dangerous}]
\par Recall that in Proposition \ref{prop:count-extremely-dangerous-intervals}, we denote by $E_q$ the union of all $q$-extremely dangerous intervals. By Lemma \ref{lm:key-lemma-extremely-dangerous}, we have that 
\[I_q \subset E_q \cup \bigcup_{i=2}^n D_q(i). \]
 By Proposition \ref{prop:count-extremely-dangerous-intervals} we have that 
\[ m(E_q) \ll R^{-\nu q}\]
for some constant $\nu >0$. On the other hand, by Lemma \ref{lm:key-lemma-extremely-dangerous-2nd-case}, we have that 
\[m(D_q(i)) \ll R^{-\nu q}\]
for any $i = 2, \dots, n$. Therefore, we have that 
\begin{align*}
F(\hat{\mathcal{I}}_{q,0}, I) R^{-q} & =  m\left(\bigcup_{I_q \in \hat{\mathcal{I}}_{q,0}} I_q\right) \\
 & \leq  m\left( E_q \bigcup_{i=2}^n D_q(i) \right) \leq m(E_q) + \sum_{i=2}^n m(D_q(i)) \ll R^{-\nu q}.
\end{align*} 
This completes the proof.
\end{proof}
\par Now we are ready to prove Proposition \ref{prop:distance-sequence-small} for $q > 10^6 n^4 N k$.
\begin{proof}[Proof of Proposition \ref{prop:distance-sequence-small} for $q > 10^6 n^4 N k$]
\par We can choose $R$ such that $R^{\nu} > 1000$. By Proposition \ref{prop:small-portion-extremely-dangerous}, we have that 
\begin{equation}
\label{equ:extremely-dangerous-case}
\left( \frac{4}{R} \right)^q F(\hat{\mathcal{I}}_{q,0}, I)  \ll  \left( \frac{4}{R} \right)^q R^{(1-\nu) q}  =  \left( \frac{4}{R^{\nu}} \right)^q <\left( \frac{4}{1000} \right)^q.
\end{equation}
\par Combining \eqref{equ:q-large-p-large}, \eqref{equ:dangerous-case} and \eqref{equ:extremely-dangerous-case}, we have that 
\[ \sum_{p = 0}^{q-1} \left( \frac{4}{R}\right)^{q-p} \max_{I_p \in \mathcal{I}_p} F(\hat{\mathcal{I}}_{q,p}, I_p) \to 0 \]
as $m \to \infty$. This proves the statement.
\end{proof}
\begin{remark}
\par In \cite{Badziahin-Harrap-Nesharim-Simmons}, Cantor winning property is introduced. It is equivalent to Cantor rich over $\R$ and is defined for higher dimensions.
\end{remark}
\begin{proof}[Proof of Theorem \ref{thm:main-task}]
\par By Definition \ref{def:m-rich}, Theorem \ref{thm:main-task} follows from Proposition \ref{prop:distance-sequence-small}.
\end{proof}
\par By Theorem \ref{thm:hausdorff-dimension-m-rich} and Theorem \ref{thm:intersection-m-rich},, Theorem \ref{thm:main-task} implies Theorem \ref{thm:main-thm} and \ref{thm:double-intersection}.
%%%%%%%%%%%%%%%%%%%%%%%%%%%%%%%%%%%%%%%%%%%%%%%%%%%%%%%%%%%%%%%%%%%%%%%%%%%%%%%%%%%%%%%%%%%%%%%%%%%%%%%%%%%%%
%%%%%%%%%%%%%%%%%%%%%%%%%%%%%%%%%%%%%%%%%%%%%%%%%%%%%%%%%%%%%%%%%%%%%%%%%%%%%%%%%%%%%%%%%%%%%%%%%%%%%%%%%%%%%
%%%%%%%%%%%%%%%%%%%%%%%%%%%%%%%%%%%%%%%%       general case       %%%%%%%%%%%%%%%%%%%%%%%%%%%%%%%%%%%%%%%%%%%
%%%%%%%%%%%%%%%%%%%%%%%%%%%%%%%%%%%%%%%%%%%%%%%%%%%%%%%%%%%%%%%%%%%%%%%%%%%%%%%%%%%%%%%%%%%%%%%%%%%%%%%%%%%%%
%%%%%%%%%%%%%%%%%%%%%%%%%%%%%%%%%%%%%%%%%%%%%%%%%%%%%%%%%%%%%%%%%%%%%%%%%%%%%%%%%%%%%%%%%%%%%%%%%%%%%%%%%%%%%
\subsection{General case}
\label{subsec-general-case}
\par Finally, let us explain how to adapt the proof for curves to handle general $C^n$ non-degenerate submanifolds.
\par Let $\bphi= \bphi(x_1, \dots, x_m): [0,1]^m \to \R^n$ be the $C^n$ differentiable map defining $\mathcal{U}$, where $m = \dim \mathcal{U}$. Then Definition \ref{def:cantor-like-construction} and Definition \ref{def:m-rich} will change according to the dimension. Intervals will be replaced by $m$-dimensional regular boxes. It is easy to see that higher dimensional versions of Theorem \ref{thm:hausdorff-dimension-m-rich} and Theorem \ref{thm:intersection-m-rich} still hold. Therefore, to prove Theorem \ref{thm:main-thm} for higher dimensional manifolds, it suffices to prove higher dimensional versions of Proposition \ref{prop:distance-sequence-small}.
\par Following the argument for curves, we split the proof into four parts: the case where $q$ is small, {\bf the generic case}, {\bf the dangerous case} and {\bf the extremely dangerous case}. When $q$ is small, we can repeat the same argument since Theorem \ref{thm:kleinbock-margulis-nondegenerate-manifold} holds for any dimension. In {\bf the generic case}, we can repeat the same argument since Thereom \ref{thm:non-divergence-dani-margulis} holds for any dimension. In {\bf the dangerous case}, we can consider $\frac{\partial \bphi}{\partial x_j}$ for $j = 1, \dots, m$ instead of $\bphi'(x)$ to prove higher dimensional versions of Proposition \ref{prop:count-dangerous-intervals} and Lemma \ref{lm:key-lemma}. Then the argument works through. In {\bf the extremely dangerous case}, we can consider partial derivatives as in {\bf the dangerous case} to prove higher dimension version of Lemma \ref{lm:key-lemma-extremely-dangerous}. Then we can repeat the same argument since higher dimensional versions of Proposition \ref{prop:count-extremely-dangerous-intervals} and Theorem \ref{thm:kleinbock-margulis-nondegenerate-manifold} still hold.
\par Combining the three cases above, we can deduce Theorem \ref{thm:main-thm} for higher dimensional $C^n$ non-degenerate submanifolds.
\medskip

\bibliography{reference.bib}{}

\newcommand{\etalchar}[1]{$^{#1}$}
\begin{thebibliography}{BFK{\etalchar{+}}12}

\bibitem[ABV18]{An-Beresnevich-Velani}
Jinpeng An, Victor Beresnevich, and Sanju Velani.
\newblock Badly approximable points on planar curves and winning.
\newblock {\em Advances in Mathematics}, 324:148--202, 2018.

\bibitem[An13]{An-2013-BLMS}
Jinpeng An.
\newblock Badziahin-{Pollington}-{Velani}'s theorem and {Schmidt}'s game.
\newblock {\em Bulletin of the London Mathematical Society}, 45(4):721--733,
  2013.

\bibitem[An16]{An}
Jinpeng An.
\newblock $2$-dimensional badly approximable vectors and {Schmidt}’s game.
\newblock {\em Duke Math. J.}, 165(2):267--284, 02 2016.

\bibitem[Ber15]{Beresnevich}
Victor Beresnevich.
\newblock Badly approximable points on manifolds.
\newblock {\em Inventiones Mathematicae}, 202(3):1199--1240, 2015.

\bibitem[BFK{\etalchar{+}}12]{broderick_fishman_kleinbock_reich_weiss_2012}
Ryan Broderick, Lior Fishman, Dmitry Kleinbock, Asaf Reich, and Barak Weiss.
\newblock The set of badly approximable vectors is strongly {C1}
  incompressible.
\newblock {\em Mathematical Proceedings of the Cambridge Philosophical
  Society}, 153(2):319–339, 2012.

\bibitem[BHNS18]{Badziahin-Harrap-Nesharim-Simmons}
Dzmitry Badziahin, Stephen Harrap, Erez Nesharim, and David Simmons.
\newblock Schmidt games and {C}antor winning sets.
\newblock {\em arXiv preprint arXiv:1804.06499}, pages 1--36, 2018.

\bibitem[BKM01]{Bernik-Kleinbock-Margulis}
Vasili Bernik, Dmitry Kleinbock, and Grigory Margulis.
\newblock Khintchine-type theorems on manifolds: the convergence case for
  standard and multiplicative versions.
\newblock {\em International Mathematics Research Notices}, 2001(9):453--486,
  2001.

\bibitem[BPV11]{BPV}
Dzmitry Badziahin, Andrew Pollington, and Sanju Velani.
\newblock On a problem in simultaneous diophantine approximation: Schmidt's
  conjecture.
\newblock {\em Annals of Mathematics}, 174(3):1837--1883, 2011.

\bibitem[Cas57]{Cassels}
J.~S. Cassels.
\newblock {\em An introduction to Diophantine approximation}.
\newblock Cambridge University Press, 1957.

\bibitem[Dan84]{Dani}
SG~Dani.
\newblock On orbits of unipotent flows on homogeneous spaces.
\newblock {\em Ergodic Theory and Dynamical Systems}, 4(01):25--34, 1984.

\bibitem[Dan86]{Dani-Bounded-Orbits}
SG~Dani.
\newblock Bounded orbits of flows on homogeneous spaces.
\newblock {\em Commentarii Mathematici Helvetici}, 61(1):636--660, 1986.

\bibitem[DM92]{Dani-Margulis}
SG~Dani and Gregory Margulis.
\newblock Limit distributions of orbits of unipotent flows and values of
  quadratic forms. {IM} {Gelfand} {Seminar}.
\newblock {\em Adv. Soviet Math}, 16:91--137, 1992.

\bibitem[DS70]{Daven_Schm}
H.~Davenport and W.~Schmidt.
\newblock Dirichlet's theorem on {D}iophantine approximation. ii.
\newblock {\em Acta Arithmetica}, 16(4):413--424, 1970.

\bibitem[EW17]{EW2017}
Manfred~Leopold Einsiedler and Thomas Ward.
\newblock {\em Functional Analysis, Spectral Theory, and Applications}, volume
  276.
\newblock Springer, 2017.

\bibitem[Jar29]{Jarnik}
Vojt{\v{e}}ch Jarn{\'\i}k.
\newblock Zur metrischen {T}heorie der {D}iophantischen approximationen.
\newblock {\em Prace Matematyczno-Fizyczne}, 36(1):91--106, 1928-1929.

\bibitem[Kle98]{Kleinbock}
Dmitry Kleinbock.
\newblock Flows on homogeneous spaces and diophantine properties of matrices.
\newblock {\em Duke Math. J.}, 95:107--124, 1998.

\bibitem[Kle08]{Kleinbock2008}
Dmitry Kleinbock.
\newblock An extension of quantitative nondivergence and applications to
  diophantine exponents.
\newblock {\em Transactions of the American Mathematical Society},
  360(12):6497--6523, 2008.

\bibitem[KM96]{Klein-Margulis-Bounded}
Dmitry Kleinbock and Gregory Margulis.
\newblock Bounded orbits of nonquasiunipotent flows on homogeneous spaces.
\newblock {\em American Mathematical Society Translations}, pages 141--172,
  1996.

\bibitem[KM98]{Klein_Mar}
Dmitry Kleinbock and Grigory Margulis.
\newblock Flows on homogeneous spaces and {Diophantine} approximation on
  manifolds.
\newblock {\em Annals of Mathematics}, pages 339--360, 1998.

\bibitem[KW08]{Klein_Weiss}
Dmitry Kleinbock and Barak Weiss.
\newblock Dirichlet's theorem on {Diophantine} approximation and homogeneous
  flows.
\newblock {\em Journal of Modern Dynamics}, 2(1):43--62, 2008.

\bibitem[KW10]{Klein-Weiss-Modified-Schmidt}
Dmitry Kleinbock and Barak Weiss.
\newblock Modified {Schmidt} games and {Diophantine} approximation with
  weights.
\newblock {\em Advances in Mathematics}, 223(4):1276--1298, 2010.

\bibitem[LM14]{Lindenstrauss-Margulis}
Elon Lindenstrauss and Gregory Margulis.
\newblock Effective estimates on indefinite ternary forms.
\newblock {\em Israel Journal of Mathematics}, 203(1):445--499, 2014.

\bibitem[Mah39]{Mahler}
Kurt Mahler.
\newblock Ein {{\"U}}bertragungsprinzip f{\"u}r lineare {U}ngleichungen.
\newblock {\em {\v{C}}asopis pro p{\v{e}}stov{\'a}n{\'\i} matematiky a fysiky},
  68(3):85--92, 1939.

\bibitem[Mah46]{Mahler1946}
Kurt Mahler.
\newblock On lattice points in n-dimensional star bodies. i. existence
  theorems.
\newblock In {\em Proceedings of the Royal Society of London A: Mathematical,
  Physical and Engineering Sciences}, volume 187, pages 151--187. The Royal
  Society, 1946.

\bibitem[MS95]{Mozes_Shah}
Shahar Mozes and Nimish Shah.
\newblock On the space of ergodic invariant measures of unipotent flows.
\newblock {\em Ergodic Theory and Dynamical Systems}, 15(01):149--159, 1995.

\bibitem[MT94]{Margulis-Tomanov}
Gregory Margulis and George Tomanov.
\newblock Invariant measures for actions of unipotent groups over local fields
  on homogeneous spaces.
\newblock {\em Inventiones mathematicae}, 116(1):347--392, 1994.

\bibitem[NS14]{Nesharim-Simmons}
Erez Nesharim and David Simmons.
\newblock Bad(s,t) is hyperplane absolute winning.
\newblock {\em Acta Arithmetica}, 164(2):145--152, 2014.

\bibitem[PV02]{Pollington-Velani}
Andrew Pollington and Sanju Velani.
\newblock On simultaneously badly approximable numbers.
\newblock {\em Journal of the London Mathematical Society}, 66(1):29--40, 2002.

\bibitem[Rat91]{Ratner}
Marina Ratner.
\newblock On {R}aghunathan's measure conjecture.
\newblock {\em Annals of Mathematics}, pages 545--607, 1991.

\bibitem[Sch66]{Schmidt-Game}
Wolfgang~M Schmidt.
\newblock On badly approximable numbers and certain games.
\newblock {\em Transactions of the American Mathematical Society},
  123(1):178--199, 1966.

\bibitem[Sch83]{Schmidt}
Wolfgang~M Schmidt.
\newblock Open problems in {D}iophantine approximation.
\newblock {\em Diophantine approximations and transcendental numbers (Luminy,
  1982)}, 31:271--287, 1983.

\bibitem[Sha09a]{Shah_2}
Nimish Shah.
\newblock Equidistribution of expanding translates of curves and
  {D}irichlet’s theorem on diophantine approximation.
\newblock {\em Inventiones Mathematicae}, 177(3):509--532, 2009.

\bibitem[Sha09b]{Shah_1}
Nimish Shah.
\newblock Limiting distributions of curves under geodesic flow on hyperbolic
  manifolds.
\newblock {\em Duke Mathematical Journal}, 148(2):251--279, 2009.

\end{thebibliography}
\bibliographystyle{alpha}
\end{document}